\documentclass[12pt,reqno]{amsart}%
\usepackage{amsfonts}
\usepackage{amsmath}
\usepackage{amssymb}
\usepackage{graphicx}%
\providecommand{\U}[1]{\protect\rule{.1in}{.1in}}
\newtheorem{theorem}{Theorem}[section]
\theoremstyle{plain}

\newtheorem{definition}{Definition}

\newtheorem{lemma}{Lemma}[section]

\newtheorem{problem}{Problem}

\newtheorem{remark}{Remark}

\numberwithin{equation}{section}
\begin{document}
\title[Concentration-compactness Principles and its application]{Concentration-compactness principle for Trudinger-Moser inequalities on
Heisenberg Groups and  existence of ground state solutions}
\author{Jungang Li}
\address{Jungang Li\\Department of Mathematics\\ University of Connecticut\\Storrs, CT 06269, USA\\E-mail: jungang.li@uconn.edu}

\author{GUOZHEN LU}
\address{Guozhen Lu\\Department of Mathematics\\University of Connecticut\\Storrs, CT 06269, USA\\E-mail: guozhen.lu@uconn.edu}
\author{MAOCHUN ZHU}
\address{Maochun Zhu\\Faculty of Science\\Jiangsu University\\Zhenjiang, 212013, China\\E-mail: zhumaochun2006@126.com }

\keywords{Trudinger-Moser inequality; Heisenberg group; Concentration-compactness Principles; Mountain-Pass theorem; exponential growth; Q-subLaplacian; ground state solution}
\thanks{The research of the second author was partly supported by a US NSF grant and a Simons Fellowship from the Simons Foundation and the research of the third author was partly supported by Natural Science Foundation of China (11601190), Natural Science Foundation of
Jiangsu Province (BK20160483) and Jiangsu University Foundation Grant (16JDG043).\\Corresponding author: Guozhen Lu, Email: guozhen.lu@uconn.edu}
%\date{ }

\subjclass[2010]{46E35;35J92; 35H20}

\begin{abstract}
Let $\mathbb{H}^{n}=\mathbb{C}^{n}\times\mathbb{R}$ be the $n$-dimensional
Heisenberg group, $Q=2n+2$ be the homogeneous dimension of $\mathbb{H}^{n}$.
We extend the well-known concentration-compactness principle on finite domains in the Euclidean spaces of \ P. L. Lions 
to the setting of the Heisenberg group $\mathbb{H}^{n}$. Furthermore, we also obtain the
corresponding concentration-compactness principle for the Sobolev space
$HW^{1,Q}\left(  \mathbb{H}^{n}\right)  $ on the entire Heisenberg group $\mathbb{H}^{n}$.

Our results improve the sharp Trudinger-Moser inequality  on domains of finite measure in $\mathbb{H}^{n}$  by Cohn and the second author  \cite{Lu}    and the corresponding one  on the whole space $\mathbb{H}^n$ by Lam and the second author
  \cite{Lam}. All the proofs of the concentration-compactness principles in the literature even in the Euclidean spaces use the rearrangement argument and the Poly\'a-Szeg\"{o} inequality. Due to the absence of the Poly\'a-Szeg\"{o} inequality   on the Heisenberg group, we will develop a different argument. Our approach is surprisingly simple and general and can be
easily applied to other settings where
symmetrization argument does not work. As an application of the concentration-compactness principle, we establish the
existence of ground state solutions for a class of $Q$- Laplacian subelliptic
equations on $\mathbb{H}^{n}:$%
\[
-\mathrm{div}\left(  \left\vert \nabla_{\mathbb{H}}u\right\vert ^{Q-2}%
\nabla_{\mathbb{H}}u\right)  +V\left(  \xi\right)  \left\vert u\right\vert
^{Q-2}u=\frac{f\left(  u\right)  }{\rho\left(  \xi\right)  ^{\beta}}%
\]
with nonlinear terms $f$ of maximal exponential growth $\exp\left(  \alpha
t^{\frac{Q}{Q-1}}\right)  $ as $t\rightarrow+\infty$.
 \end{abstract}

\maketitle
\section{Introduction}

Let $\Omega\subseteq$ $%
%TCIMACRO{\U{211d} }%
%BeginExpansion
\mathbb{R}
%EndExpansion
^{n}$ and $W_{0}^{1,q}\left(  \Omega\right)  $ be the usual Sobolev space,
that it, the completion of $C_{0}^{\infty}\left(  \Omega\right)  $ with the
norm%
\[
\left\Vert u\right\Vert _{W^{1,q}\left(  \Omega\right)  }=\left(  \int
_{\Omega}\left(  \left\vert u\right\vert ^{q}+\left\vert \nabla u\right\vert
^{q}\right)  dx\right)  ^{\frac{1}{q}}.
\]
If $1\le q<n$, the classical Sobolev embedding says that $W_{0}^{1,q}\left(
\Omega\right)  \hookrightarrow L^{s}\left(  \Omega\right)  $ for $1\leq s\leq
q^{\ast}$, where $q^{\ast}:=\frac{nq}{n-q}$. When $q=n$, it is known that
\[
W_{0}^{1,n}\left(  \Omega\right)  \hookrightarrow L^{s}\left(  \Omega\right)
\text{ for any }n\leq s<+\infty\text{, }%
\]
but $\ W_{0}^{1,n}\left(  \Omega\right)  \varsubsetneq L^{\infty}\left(
\Omega\right)  $. When $\Omega$ is of finite measure, the analogue of the Sobolev embedding is the well-known
Trudinger's inequality, which was established independently by Yudovi\v{c} \cite{Yu}, Poho\v{z}aev \cite{Po}, and Trudinger \cite{Tru}.  In 1971, Moser   sharpened in \cite{moser} Trudinger's inequality, and proved the following  inequality:
\begin{equation}
\underset{\left\Vert \nabla u\right\Vert _{L^n\left(  \Omega\right)  }\leq
1}{\underset{u\in W_{0}^{1,n}\left(  \Omega\right)  }{\sup}}\int_{\Omega
}e^{\alpha\left\vert u\right\vert ^{\frac{n}{n-1}}}dx<\infty\text{ iff
}\alpha\leq\alpha_{n}=n\omega_{n-1}^{\frac{1}{n-1}}, \label{moser-tru}%
\end{equation}
where $\omega_{n-1}$ is the $n-1$ dimensional surface measure of the unit ball
in $%
%TCIMACRO{\U{211d} }%
%BeginExpansion
\mathbb{R}
%EndExpansion
^{n}$ and $|\Omega|<\infty$. Inequality (\ref{moser-tru}) is known as the \textit{Trudinger-Moser inequality}.
In 1985, Lions \cite{lions} established the
Concentration-Compactness Principle associated with (\ref{moser-tru}), which
tells us that, if $\left\{  u_{k}\right\}  $ is a sequence of functions in
$W_{0}^{1,n}\left(  \Omega\right)  $\ with $\left\Vert \nabla u_{k}\right\Vert
_{n}=1$ such that $u_{k}\rightharpoonup u$ weakly in $W^{1,n}\left(
\Omega\right)  $, then for any $0<p<M_{n,u}:=\left(  1-\left\Vert \nabla
u\right\Vert _{n}^{n}\right)  ^{-1/\left(  n-1\right)  }$, one has%
\begin{equation}
\underset{k}{\sup}\int_{\Omega}e^{\alpha_{n}p\left\vert u_{k}\right\vert
^{\frac{n}{n-1}}}dx<\infty. \label{lions}%
\end{equation}
This conclusion gives more precise information and is stronger than (\ref{moser-tru}) when
$u_{k}\rightharpoonup u\neq0$ weakly in $W_{0}^{1,n}\left(  \Omega\right)  $.

\medskip

When $\left\vert \Omega\right\vert =+\infty$, the inequality (\ref{moser-tru})
is meaningless. In this case, the first related inequalities have been
considered by D.M. Cao \cite{cao} in the case $N=2$ and for any dimension by
do \'{O} \cite{J.M. do1} and Adachi-Tanaka \cite{Adachi-Tanaka}. For two-weighted subcritical Trudinger-Moser inequalities, see \cite{INW, DongLu}. 
 Note that,
unlike  (\ref{moser-tru}), all these results have been proved in the
subcritical growth case, that is $\alpha<\alpha_{n}$. In \cite{ruf}, Ruf
showed that in the case $N=2$, the exponent $\alpha_{2}=4\pi$ becomes
admissible if the Dirichlet norm $\int_{\Omega}\left\vert \nabla u\right\vert
^{2}dx$ is replaced by $W^{1,2}$ norm $\int_{\Omega}\left(  \left\vert
u\right\vert ^{2}+\left\vert \nabla u\right\vert ^{2}\right)  dx$. Later, Y.X.
Li and Ruf \cite{liruf} established the same critical inequality  as in \cite{ruf} in arbitrary dimensions.
These critical and subcritical inequalities have been proved to be equivalent  in \cite{LLZ}.

\medskip

While there has been much progress for Trudinger-Moser type inequalities and the concentration-compactness phenomenon on the Euclidean spaces, much less is known on the Heisenberg group.   We recall that most
of the proofs for  Trudinger-Moser inequalities in the Euclidean space are
based on the rearrangement argument. When one considers the  Trudinger-Moser
inequalities in the\ subelliptic setting, one often attempts to use the radial
non-increasing rearrangement $u^{\ast}$ of functions $u$. Unfortunately, it is
not true that the $L^{p}$ norm of the subelliptic gradient of the
rearrangement of a function is dominated by the $L^{p}$ norm of the
subelliptic gradient of the function. In other words, the
P\'{o}lya-Szeg\"{o}\ type inequality in the subelliptic setting like%
\begin{equation}
\left\Vert \nabla_{\mathbb{H}^n}u^{\ast}\right\Vert _{L^{p}}\leq\left\Vert
\nabla_{\mathbb{H}^n}u\right\Vert _{L^{p}  } \label{poly}%
\end{equation}
is not available. Actually, from the  work of D. Jerison and J. Lee \cite{Jerison} on sharp $L^2$ to $L^{\frac{2Q}{Q-2}}$ inequality on the Heisenberg group with applications to the solution to the CR Yamabe problem, we
know that this inequality fails to hold for the case $p=2$ in Heisenberg groups.

\medskip

The sharp Trudinger-Moser inequality on Heisenberg groups was due to
Cohn and the second author \cite{Lu}  and has been extended to  the Heisenberg type groups and   Carnot groups in \cite{Lu2} and \cite{Balogh}  and  with  singular weights  in   \cite{Lam2}. 
  Furthermore,  Lam and the
second author developed in \cite{Lam, Lam1} a rearrangement-free argument by considering the level sets of the functions
under consideration, this argument enables them to deduce the global Trudinger-Moser inequalities on the entire space
from the local ones on the level sets (see also \cite{lu-yang1} for adaptation of such an argument). Therefore, both sharp critical and subcritical Trudinger-Moser inequalities are established on the entire Heisenberg group in \cite{Lam, Lam3}.

\medskip

More recently, \v{C}ern\'{y} et al. in \cite{Cerny} discover a new approach to
obtain and sharpen Lions's concentration compactness principles (\ref{lions}) as well as fill in a gap in \cite{lions}. This
approach was further extended  to study the Concentration-compactness
principle for the whole space $\mathbb{R}^{n}$ by do \'{O} et al. in
\cite{do}. Their results can be stated as follows: let $\left\{
u_{k}\right\}  $ be a sequence of functions in $W_{0}^{1,n}\left(
\mathbb{R}^{n}\right)  $\ with $\left\Vert u_{k}\right\Vert _{W^{1,n}\left(
\mathbb{R}^{n}\right)  }=1$ such that $u_{k}\rightharpoonup u$ weakly in
$W^{1,n}\left(  \mathbb{R}^{n}\right)  $, then for any $0<p<\tilde{M}%
_{n,u}:=\left(  1-\left\Vert u\right\Vert _{W^{1,n}\left(  \mathbb{R}%
^{n}\right)  }\right)  ^{-1/\left(  n-1\right)  }$,%
\begin{equation}
\underset{k}{\sup}\int_{%
%TCIMACRO{\U{211d} }%
%BeginExpansion
\mathbb{R}
%EndExpansion
^{n}}e^{\alpha_{n}p\left\vert u_{k}\right\vert ^{\frac{n}{n-1}}}dx<\infty.
\label{do}%
\end{equation}
Furthermore, $\tilde{M}_{n,u}$ is sharp in the sense that there exists a
sequence $\left\{  u_{k}\right\}  $ satisfying $\left\Vert u_{k}\right\Vert
_{W^{1,n}\left(  \mathbb{R}^{n}\right)  }=1$ and $u_{k}\rightharpoonup u$
weakly in $W^{1,n}\left(  \mathbb{R}^{n}\right)  $ such that the supremum
(\ref{do}) is infinite for $p\geq\tilde{M}_{n,u}$\footnote{The sequence
$\left\{  u_{k}\right\}  $ constructed in \cite{do} cannot show that the
supremum (\ref{do}) is infinite for $p=\tilde{M}_{n,u}$ (see Remark \ref{sharpness})}. We also note a recent work 
on sharp Trudinger-Moser type inequalities in the spirit of Lions' work on the whole spaces \cite{LamLuTang}. 

Nevertheless, we mention that arguments of \cite{Cerny}  and \cite{do} still rely on the Poly\'a-Szeg\"{o} inequality in the Euclidean spaces and such an inequality is not available in the subelliptic setting. 

\medskip

Now, it is fairly natural to ask whether the Concentration-compactness
principles (\ref{lions}) and (\ref{do}) still holds for the subelliptic
setting in spite of its absence of the Poly\'a-Szeg\"{o} inequality in such a setting. In this paper, we will give an affirmative answer to this question.
More precisely, we first prove a concentration-compactness principle for
domains with finite measure on Heisenberg groups (Theorem \ref{Theorem1}), and
then prove the concentration-compactness principle for the Horizontal Sobolev
space $HW^{1,Q}\left(  \mathbb{H}^{n}\right)  $--Theorem \ref{Theorem2} (for
definition of $HW^{1,Q}\left(  \mathbb{H}^{n}\right)  $ see Section
\ref{Hesen}). Theorem \ref{Theorem1} sharpens the Trudinger-Moser inequality
by Cohn and the second author \cite{Lu} and recent one of Lam et al.
\cite{Lam2}, Theorem \ref{Theorem2} improves the sharp Trudinger-Moser
inequality by Lam and the second author \cite{Lam}.

\medskip

In the proof of the Concentration-compactness principles on Heisenberg groups,
we carry out a different  argument from \cite{Cerny}
and \cite{do}. It is worthwhile to note that our approach   can be easily applied to the other subelliptic setting
such as Carnot groups with virtually no modifications.

\medskip As an application of Concentration-compactness principles on
Heisenberg groups, we study the existence of positive ground state solution to
a class of partial differential equations with exponential growth on
$\mathbb{H}^{n}$ of the form:%
\begin{equation}
-\mathrm{div}\left(  \left\vert \nabla_{\mathbb{H}}u\right\vert ^{Q-2}%
\nabla_{\mathbb{H}}u\right)  +V\left(  \xi\right)  \left\vert u\right\vert
^{Q-2}u=\frac{f\left(  u\right)  }{\rho\left(  \xi\right)  ^{\beta}%
}\label{eqa}%
\end{equation}
for any $0\leq\beta<Q$, where $V:\mathbb{H}^{n}\rightarrow\mathbb{R}$ is a
continuous potential, and $f:\mathbb{R}\rightarrow\mathbb{R}$ behaves like
$\exp\left(  \alpha t^{\frac{Q}{Q-1}}\right)  $ when $t\rightarrow\infty$ (for
the meaning of $\nabla_{\mathbb{H}}$ and $\rho\left(  \xi\right)  $ see
Section \ref{Hesen}).

\medskip

We remark that the Trudinger-Moser type inequalities play an
important role in the study of the existence of solutions to nonlinear partial
differential equations of exponential growth in Euclidean spaces. A good deal of works  have been done and  we just quote
some of them  on this subject, which are a good starting point for
further bibliographic references:
\cite{Adimurthi2,Adimurthi1,  CC, Figueiredo1,Figueiredo,J.M. do1,J. M.
do,do, Flucher, lin, Lam4, LiY,  lu-yang, Malchiodi, struwe, yang, zhu}, etc. 
  
Existence and multiplicity of nontrivial nonnegative solutions to the
equations (\ref{eqa}) on the Heisenberg groups have been proved   in a series of papers 
\cite{Lu4, Lam, Lam2, Lam3}. In their argument, they apply the Trudinger-Moser inequality in
the whole space $\mathbb{H}^{n}$  (Lemma \ref{Lemma2} in Section \ref{Hesen})
combined with mountain-pass theorem, minimization and Ekelands variational
principle. 
 Nevertheless,   the existence of {\it ground state} solutions to the sub-elliptic equation  (\ref{eqa}) on the Heisenberg groups  has  not been established yet so far. The concentration-compactness principles on
Heisenberg groups proved in this paper makes it possible to establish such an existence result.

\medskip

This paper is organized as follows: in Section \ref{Hesen} we recall some
basic facts about Heisenberg Groups and state precisely our main results; in
Section \ref{proof of CC} we first prove the concentration compactness
principles for  Trudinger-Moser inequalities on domains with finite measure --
Theorem \ref{Theorem1}, and then we give the proof for the concentration
compactness principles for Horizontal Sobolev space $HW^{1,Q}\left(
\mathbb{H}^{n}\right)  $ --Theorem \ref{Theorem2}. As an application, in
Section \ref{pde}, we consider the equations (\ref{eqa}) and establish the existence of the {\it ground state solutions} and \ prove Theorem
\ref{ground state} by using the minimax argument and Theorem \ref{Theorem2}.

\section{\bigskip Preliminaries and statement of the results\label{Hesen}}

\subsection{Background on Heisenberg groups}

Let $\mathbb{H}^{n}=\mathbb{C}^{n}\times\mathbb{R}$ be the $n$-dimensional
Heisenberg group, whose group structure is given by
\[
\left(  x,t\right)  \circ\left(  x^{\prime},t^{\prime}\right)  =\left(
x+x^{\prime},t+t^{\prime}+2\mathrm{im}\left(  x\cdot\bar{x}^{\prime}\right)
\right)  .
\]
The Lie algebra of $\mathbb{H}^{n}$ is generated by the left invariant vector
fields
\[
X_{i}=\frac{\partial}{\partial x_{i}}+2y_{i}\frac{\partial}{\partial t}%
,Y_{i}=\frac{\partial}{\partial y_{i}}-2x_{i}\frac{\partial}{\partial
t},T=\frac{\partial}{\partial t},
\]
for $i=1,\ldots,n$. These generators satisfy the non-commutative relationship
$\left[  X_{i},Y_{i}\right]  =4\delta_{ij}T$. Moreover, all the commutators of
length greater than two vanish, and thus this is a nilpotent, graded, and
stratified group of step two. \

\medskip

For each real number $r\in%
%TCIMACRO{\U{211d} }%
%BeginExpansion
\mathbb{R}
%EndExpansion
$, there is a dilation naturally associated with the Heisenberg group
structure which is usually denoted as $\delta_{r}\left(  x,t\right)  =\left(
rx,r^{2}t\right)  $. The Jacobian determinant of $\delta_{r}$ is $r^{Q}$,
where $Q=2n+2$ is the homogeneous dimension of $\mathbb{H}^{n}$.

\medskip

We will use $\xi=(x,t)$ to denote any point $(x,t)\in\mathbb{H}^{n}$, then the
anisotropic dilation structure on $\mathbb{H}^{n}$ introduces a homogeneous
norm $\left\vert \xi\right\vert =\left(  \left\vert x\right\vert ^{4}%
+t^{2}\right)  ^{1/4}$. Let
\[
B_{r}=\left\{  \xi:\left\vert \xi\right\vert <r\right\}
\]
be the metric ball of center $0$ and radius $r$ in $\mathbb{H}^{n}$. Since the
Lebesgue measure in $\mathbb{R}^{2n+1}$ is the Haar measure on $\mathbb{H}%
^{n}$, one has (writing $\left\vert A\right\vert $ for the measure of $A$)
\[
\left\vert B_{r}\right\vert =\omega_{Q}r^{Q},
\]
where $\omega_{Q}$ is a positive constant only depending on $Q$
(see \cite{Lu}).

\medskip

We write $\left\vert \nabla_{\mathbb{H}}u\right\vert $ to express the norm of
the subelliptic gradient of the function $u:\mathbb{H}^{n}\rightarrow
\mathbb{R}:$%
\[
\left\vert \nabla_{\mathbb{H}}u\right\vert =\sqrt{\sum\left(
X_{i}u\right)  ^{2}+\left(  Y_{i}u\right)  ^{2}}.
\]
Let $\Omega$ be an open set in $\mathbb{H}^{n}$ and $p>1$. We define
the\ Horizontal Sobolev Spaces
\[
HW^{1,p}\left(  \Omega\right)  =\left\{  u\in L^{p}\left(  \Omega\right)
:\left\Vert u\right\Vert _{HW^{1,p}\left(  \Omega\right)  }<\infty\right\}
\]
with the norm
\[
\left\Vert u\right\Vert _{HW^{1,p}\left(  \Omega\right)  }=\left(
\int_{\Omega}\left(  \left\vert \nabla_{\mathbb{H}}u\left(  z,t\right)
\right\vert ^{p}+\left\vert u\left(  z,t\right)  \right\vert ^{p}\right)
dxdt\right)  ^{1/p}.
\]
Also, we define the space $HW_{0}^{1,p}\left(  \Omega\right)  $ as the closure
of $C_{0}^{\infty}\left(  \Omega\right)  $ in the norm of $HW^{1,p}\left(
\Omega\right)  $.

\subsection{Some useful known results on Heisenberg groups}

In this subsection, we collect some known results which will be used in the following.

\begin{lemma}[\cite{Lu}]
\label{lu}Let $\rho=\left\vert \xi\right\vert $ be the homogeneous
norm of the element $\xi=\left(  x,t\right)  \in\mathbb{H}^{n}$, and $g\left(
\xi\right)  =g\left(  \rho\right)  $ be a $C^{1}$ radial function on
$\mathbb{H}^{n}$. Then
\[
\left\vert \nabla_{\mathbb{H}}g\left(  \xi\right)  \right\vert =\frac
{g^{\prime}\left(  \rho\right)  }{\rho}\left\vert x\right\vert .
\]

\end{lemma}

\begin{lemma} [\cite{Lam2}]
\label{Lemma1} Let $\alpha_{Q}=Q\left(  2\pi^{n}\Gamma\left(
\frac{1}{2}\right)  \Gamma\left(  \frac{Q-1}{2}\right)  \Gamma\left(  \frac
{Q}{2}\right)  ^{-1}\Gamma\left(  n\right)  ^{-1}\right)  ^{\frac{1}{Q-1}},$
$0\leq\beta<Q$. There exists a uniform constant $c$ depending only on
$Q,\beta$ such that for all $\Omega\subset\mathbb{H}^{n}$ with $\left\vert
\Omega\right\vert <\infty$ and $\alpha\leq\alpha_{Q,\beta}=\alpha_{Q}\left(
1-\frac{\beta}{Q}\right)  $, one has%
\begin{equation}
\underset{\left\Vert \nabla_{\mathbb{H}}u\right\Vert _{L^{Q}}\leq1}%
{\underset{u\in HW_{0}^{1,Q}\left(  \Omega\right)  }{\sup}}\int_{\Omega}%
\frac{\exp\left(  \alpha u\left(  \xi\right)  ^{\frac{Q}{Q-1}}\right)  }%
{\rho\left(  \xi\right)  ^{\beta}}d\xi<c.\label{moser-lu}%
\end{equation}
The constant $\alpha_{Q,\beta}$ is the best possible in the sense that if
$\alpha>$ $\alpha_{Q,\beta}$, then the supremum above is infinite.
\end{lemma}

\begin{lemma}[\cite{Lam}]
\label{Lemma2} Let $0\leq\beta<Q$. There exists a uniform constant
$c$ depending only on $Q,\beta$ such that for all $\alpha\leq\alpha_{Q,\beta}%
$, one has%
\begin{equation}
\underset{\left\Vert f\right\Vert _{HW^{1,Q}\left(  \mathbb{H}^{n}\right)
}\leq1}{\underset{f\in HW^{1,Q}\left(  \mathbb{H}^{n}\right)  }{\sup}}%
\int_{\mathbb{H}^{n}}\frac{\Phi\left(  \alpha f\left(  \xi\right)  ^{\frac
{Q}{Q-1}}\right)  }{\rho\left(  \xi\right)  ^{\beta}}d\xi<c. \label{unbounded}%
\end{equation}
where $\Phi\left(  t\right)  =e^{t}-\underset{j=0}{\overset{Q-2}{\sum}}%
\frac{t^{j}}{j!}$.\ The constant $\alpha_{Q,\beta}$ is the best possible in
the sense that if $\alpha>\alpha_{Q,\beta}$, then the supremum in the
inequality (\ref{moser-lu}) is infinite.
\end{lemma}

\subsection{\bigskip Statement of the main results}

Now, we are ready to state precisely the main results of this paper.

\begin{theorem}
[Concentration compactness for domains with finite measure]\label{Theorem1}Let
$0\leq\beta<Q$. Assume that $\left\{  u_{k}\right\}  $ is  a sequence in
$HW_{0}^{1,Q}\left(  \Omega\right)  $ with $\left\vert \Omega\right\vert
<\infty$, such that $\left\Vert \nabla_{\mathbb{H}}u_{k}\right\Vert _{Q}=1$
and $u_{k}\rightharpoonup u\neq0$ in $HW_{0}^{1,Q}\left(  \Omega\right)  $.
If
\[
0<p<M_{Q,u}:=\frac{1}{\left(  1-\left\Vert \nabla_{\mathbb{H}}u\right\Vert
_{Q}^{Q}\right)  ^{1/\left(  Q-1\right)  }}%
\]
then%
\[
\underset{k}{\sup}\int_{\Omega}\frac{e^{\alpha_{Q,\beta}pu_{k}^{\frac{Q}{Q-1}%
}}}{\rho\left(  \xi\right)  ^{\beta}}d\xi<\infty.
\]
Moreover, $M_{Q,u}$ is sharp in the sense that there exists a sequence
$\left\{  u_{k}\right\}  $ satisfying $\left\Vert \nabla_{\mathbb{H}%
}u\right\Vert _{Q}^{Q}=1$ and $u_{k}\rightharpoonup u\neq0$ in $HW_{0}%
^{1,Q}\left(  \Omega\right)  $ such that the supremum is infinite for $p\geq
M_{Q,u}$.
\end{theorem}

\begin{theorem}
[Concentration compactness for $HW^{1,Q}\left(  \mathbb{H}^{n}\right)  $%
]\label{Theorem2}Let $0\leq\beta<Q$. \bigskip Assume that $\left\{
u_{k}\right\}  $ is a sequence in $HW^{1,Q}\left(  \mathbb{H}^{n}\right)  $
such that $\left\Vert u_{k}\right\Vert _{HW^{1,Q}\left(  \mathbb{H}%
^{n}\right)  }^{Q}=1$ and $u_{k}\rightharpoonup u\neq0$ in $HW^{1,Q}\left(
\mathbb{H}^{n}\right)  $. If
\[
0<p<\tilde{M}_{Q,u}:=\frac{1}{\left(  1-\left\Vert u\right\Vert _{HW^{1,Q}%
\left(  \mathbb{H}^{n}\right)  }^{Q}\right)  ^{1/\left(  Q-1\right)  }},
\]
then%
\begin{equation}
\underset{k}{\sup}\int_{\mathbb{H}^{n}}\frac{\Phi\left(  \alpha_{Q,\beta
}pu_{k}^{\frac{Q}{Q-1}}\right)  }{\rho\left(  \xi\right)  ^{\beta}}d\xi
<\infty,\label{suprem}%
\end{equation}
where $\Phi\left(  t\right)  =e^{t}-\underset{j=0}{\overset{Q-2}{\sum}}%
\frac{t^{j}}{j!}$. Furthermore, $\tilde{M}_{Q,u}$ is sharp in the sense that
there exists a sequence $\left\{  u_{k}\right\}  $ satisfying $\left\Vert
u_{k}\right\Vert _{HW^{1,Q}\left(  \mathbb{H}^{n}\right)  }^{Q}=1$ and
$u_{k}\rightharpoonup u\neq0$ in $HW^{1,Q}\left(  \mathbb{H}^{n}\right)  $
such that the supremum is infinite for $p>\tilde{M}_{Q,u}.$
\end{theorem}
The following natural question  still remains open at this time.

\begin{problem} Does (\ref{suprem}) still hold when $p=\tilde{M}_{Q,u}$?
\end{problem}
Now, Let us give the definition of the ground state\ solution of (\ref{eqa}):

\begin{definition}
[Ground state\ solution]A function $u$ is said to be the ground
state\ solution of (\ref{eqa}), if $u$ is positive and minimizes the energy
functional
 associated
 to the equation (\ref{eq}) defined by%
\begin{equation*}
\newline J\left(  u\right)  =\frac{1}{Q}\int_{\mathbb{H}^{n}}\left(
\left\vert \nabla_{\mathbb{H}}u\right\vert ^{Q}+V\left(  \xi\right)
\left\vert u\right\vert ^{Q}\right)  d\xi-\int_{\mathbb{H}^{n}}\frac{F\left(
u\right)  }{\rho\left(  \xi\right)  ^{\beta}}d\xi
\end{equation*}
 within the set of nontrivial solutions of (\ref{eqa}).
\end{definition}
For the equation (\ref{eq}), we obtain the following
\begin{theorem}
\label{ground state} Under the hypotheses of (H1) and (H2) in Section
\ref{pde}, the \ $Q-$sub-Laplacian equations (\ref{eqa}) has a positive ground
state solution.
\end{theorem}

\medskip Throughout this paper, denote by the letter $c$ some positive
constant which may vary from line to line.

\section{Concentration-Compactness principles on Heisenberg
groups\label{proof of CC}}

\subsection{Concentration-Compactness principle for domains with finite
measure\label{domain finite}}

In this subsection, we give the

\begin{proof}
[Proof of Theorem \ref{Theorem1}]Since $\left\Vert \nabla_{\mathbb{H}%
}u\right\Vert _{Q}\leq\underset{k}{\lim}\left\Vert \nabla_{\mathbb{H}}%
u_{k}\right\Vert _{Q}=1$, we split the proof into two cases. \

Case 1: \ $\left\Vert \nabla_{\mathbb{H}}u\right\Vert _{Q}<1.$\ We assume by
contradiction for some $p_{1}<M_{Q,u}$, we have%
\[
\underset{k}{\sup}\int_{\Omega}\frac{\exp\left(  \alpha_{Q,\beta}p_{1}%
u_{k}^{\frac{Q}{Q-1}}\right)  }{\rho\left(  \xi\right)  ^{\beta}}d\xi
=+\infty.
\]
Set
\[
\Omega_{L}^{k}=\left\{  \xi\in\Omega: u_{k}\left(  \xi\right)  \geq L\right\}
,
\]
where $L$ is some constant. \ Let $v_{k}=u_{k}-L$. Then for any $\varepsilon
>0$, one has
\begin{equation}
u_{k}^{\frac{Q}{Q-1}}\leq\left(  1+\varepsilon\right)  v_{k}^{\frac{Q}{Q-1}%
}+C\left(  \varepsilon,Q\right)  L^{\frac{Q}{Q-1}}. \label{2}%
\end{equation}

Since $0\leq\beta<Q$, we have
\begin{align*}
\int_{\Omega}\frac{\exp\left(  \alpha_{Q,\beta}p_{1}u_{k}^{\frac{Q}{Q-1}%
}\right)  }{\rho\left(  \xi\right)  ^{\beta}}d\xi &  =\int_{\Omega_{L}^{k}%
}\frac{\exp\left(  \alpha_{Q,\beta}p_{1}u_{k}^{\frac{Q}{Q-1}}\right)  }%
{\rho\left(  \xi\right)  ^{\beta}}d\xi+\int_{\Omega\backslash\Omega_{L}^{k}%
}\frac{\exp\left(  \alpha_{Q,\beta}p_{1}u_{k}^{\frac{Q}{Q-1}}\right)  }%
{\rho\left(  \xi\right)  ^{\beta}}d\xi\\
&  \leq\int_{\Omega_{L}^{k}}\frac{\exp\left(  \alpha_{Q,\beta}p_{1}%
u_{k}^{\frac{Q}{Q-1}}\right)  }{\rho\left(  \xi\right)  ^{\beta}}d\xi
+c\exp\left(  \alpha_{Q,\beta}p_{1}L^{\frac{Q}{Q-1}}\right)  \int_{\Omega
}\frac{1}{\rho\left(  \xi\right)  ^{\beta}}d\xi\\
&  \leq\int_{\Omega_{L}^{k}}\frac{\exp\left(  \alpha_{Q,\beta}p_{1}%
u_{k}^{\frac{Q}{Q-1}}\right)  }{\rho\left(  \xi\right)  ^{\beta}}d\xi+c\left(
L,Q, |\Omega|, \beta\right)  ,
\end{align*}
and then
\[
\underset{k}{\sup}\int_{\Omega_{L}^{k}}\frac{\exp\left(  \alpha_{Q,\beta}%
p_{1}u_{k}^{\frac{Q}{Q-1}}\right)  }{\rho\left(  \xi\right)  ^{\beta}}%
d\xi=\infty.
\]

By (\ref{2}) we have
\[
\int_{\Omega_{L}^{k}}\frac{\exp\left(  \alpha_{Q,\beta}p_{1}u_{k}^{\frac
{Q}{Q-1}}\right)  }{\rho\left(  \xi\right)  ^{\beta}}d\xi\leq\exp\left(
\alpha_{Q,\beta}p_{1}C\left(  \varepsilon,Q\right)  L^{\frac{Q}{Q-1}}\right)
\cdot\int_{\Omega_{L}^{k}}\frac{\exp\left(  \left(  1+\varepsilon\right)
\alpha_{Q,\beta}p_{1}v_{k}^{\frac{Q}{Q-1}}\right)  }{\rho\left(  \xi\right)
^{\beta}}d\xi.
\]
Thus
\[
\underset{k}{\sup}\int_{\Omega_{L}^{k}}\frac{\exp\left(  \bar{p}_{1}%
\alpha_{Q,\beta}v_{k}^{\frac{Q}{Q-1}}\right)  }{\rho\left(  \xi\right)
^{\beta}}d\xi=\infty,
\]
where $\bar{p}_{1}=\left(  1+\varepsilon\right)  p_{1}<M_{Q,u}.$

Now, we define
\[
T^{L}\left(  u\right)  =\min\left\{  L,u\right\}  \text{ and }T_{L}\left(
u\right)  =u-T^{L}\left(  u\right)
\]
and choose $L$ such that%
\begin{equation}
\frac{1-\left\Vert \nabla_{\mathbb{H}}u\right\Vert _{Q}^{Q}}{1-\left\Vert
\nabla_{\mathbb{H}}T^{L}u\right\Vert _{Q}^{Q}}>\left(  \frac{\bar{p}_{1}%
}{M_{Q,u}}\right)  ^{Q-1}. \label{4}%
\end{equation}
We claim that
\[
\underset{k}{\lim\sup}\int_{\Omega_{L}^{k}}\left\vert \nabla_{\mathbb{H}}%
v_{k}\right\vert ^{Q}d\xi<\left(  \frac{1}{\bar{p}_{1}}\right)  ^{Q-1}.
\]
If not, then up to a subsequence, one has
\begin{equation}
\int_{\Omega_{L}^{k}}\left\vert \nabla_{\mathbb{H}}v_{k}\right\vert ^{Q}%
d\xi=\int_{\Omega}\left\vert \nabla_{\mathbb{H}}T_{L}u_{k}\right\vert ^{Q}%
d\xi\geq\left(  \frac{1}{\bar{p}_{1}}\right)  ^{Q-1}+o_{k}\left(  1\right)  .
\label{3}%
\end{equation}
\ Thus,
\begin{align*}
\left(  \frac{1}{\bar{p}_{1}}\right)  ^{Q-1}+\int_{\Omega}\left\vert
\nabla_{\mathbb{H}}T^{L}u_{k}\right\vert ^{Q}d\xi+o_{k}\left(  1\right)   &
\leq\int_{\Omega}\left\vert \nabla_{\mathbb{H}}T_{L}u_{k}\right\vert ^{Q}%
d\xi+\int_{\Omega\backslash\Omega_{L}^{k}}\left\vert \nabla_{\mathbb{H}}%
u_{k}\right\vert ^{Q}d\xi\\
&  =\int_{\Omega_{L}^{k}}\left\vert \nabla_{\mathbb{H}}u_{k}\right\vert
^{Q}d\xi+\int_{\Omega\backslash\Omega_{L}^{k}}\left\vert \nabla_{\mathbb{H}%
}u_{k}\right\vert ^{Q}d\xi=1.
\end{align*}
For $L>0$ fixed, $T^{L}u_{k}$ is also bounded in $HW^{1,Q}\left(
\Omega\right)  $. Hence, up to a subsequence, $T^{L}u_{k}\rightharpoonup
T^{L}u$ in $HW^{1,Q}\left(  \Omega\right)  $ and $T^{L}u_{k}\rightarrow
T^{L}u$ almost everywhere in $\Omega$. By the lower semicontinuity of the norm
in $HW^{1,Q}\left(  \Omega\right)  $ and the above inequality, we have%
\[
\bar{p}_{1}\geq\frac{1}{\left(  1-\underset{k\rightarrow\infty}{\lim\inf
}\left\Vert \nabla_{\mathbb{H}}T^{L}u_{k}\right\Vert _{Q}^{Q}\right)
^{\frac{1}{Q-1}}}\geq\frac{1}{\left(  1-\left\Vert \nabla_{\mathbb{H}}%
T^{L}u\right\Vert _{Q}^{Q}\right)  ^{\frac{1}{Q-1}}},
\]
combining with (\ref{4}), we derive
\[
\bar{p}_{1}\geq\frac{1}{\left(  1-\left\Vert \nabla_{\mathbb{H}}%
T^{L}u\right\Vert _{Q}^{Q}\right)  ^{\frac{1}{Q-1}}}>\frac{\bar{p}_{1}%
}{M_{Q,u}}\frac{1}{\left(  1-\left\Vert \nabla_{\mathbb{H}}u\right\Vert
_{Q}^{Q}\right)  ^{\frac{1}{Q-1}}}=\bar{p}_{1},
\]
which is a contradiction. \ Therefore
\[
\underset{k}{\lim\sup}\int_{\Omega_{L}^{k}}\left\vert \nabla_{\mathbb{H}}%
v_{k}\right\vert ^{Q}d\xi<\left(  \frac{1}{\bar{p}_{1}}\right)  ^{Q-1}.
\]
By the Trudinger-Moser inequality (\ref{moser-lu}), we derive $\ \ $%
\[
\underset{k}{\sup}\int_{\Omega_{L}^{k}}\frac{\exp\left(  \bar{p}_{1}%
\alpha_{Q,\beta}v_{k}^{\frac{Q}{Q-1}}\right)  }{\rho\left(  \xi\right)
^{\beta}}d\xi<\infty,
\]
which is also a contradiction. The proof is finished in this case.

\medskip

Case 2: $\left\Vert \nabla_{\mathbb{H}}u\right\Vert _{Q}=1$. We can iterate
the procedure as in Case $1$ and get
\[
\underset{k}{\sup}\int_{\Omega_{L}^{k}}\frac{\exp\left(  \bar{p}_{1}%
\alpha_{Q,\beta}v_{k}^{\frac{Q}{Q-1}}\right)  }{\rho\left(  \xi\right)
^{\beta}}=\infty,
\]
where $\bar{p}_{1}=\left(  1+\varepsilon\right)  p_{1}$. \ Then we have
\[
\underset{k}{\lim\sup}\int_{\Omega_{L}^{k}}\left\vert \nabla_{\mathbb{H}}%
v_{k}\right\vert ^{Q}d\xi=\underset{k}{\lim\sup}\int_{\Omega}\left\vert
\nabla_{\mathbb{H}}T_{L}u_{k}\right\vert ^{Q}d\xi\geq\left(  \frac{1}{\bar
{p}_{1}}\right)  ^{Q-1},
\]
thus,
\[
\left\Vert \nabla_{\mathbb{H}}T^{L}u\right\Vert _{Q}^{Q}\leq\underset{k}%
{\lim\inf}\int_{\Omega}\left\vert \nabla_{\mathbb{H}}T^{L}u_{k}\right\vert
^{Q}d\xi=1-\underset{k}{\lim\sup}\int_{\Omega}\left\vert \nabla_{\mathbb{H}%
}T_{L}u_{k}\right\vert ^{Q}d\xi\leq1-\left(  \frac{1}{\bar{p}_{1}}\right)
^{Q-1}.
\]
On the other hand, since $\left\Vert \nabla_{\mathbb{H}}u\right\Vert _{Q}=1$,
we can take $L$ large enough such that
\[
\left\Vert \nabla_{\mathbb{H}}T^{L}u\right\Vert _{Q}^{Q}>1-\frac{1}{2}\left(
\frac{1}{\bar{p}_{1}}\right)  ^{Q-1},
\]
which is contradiction, and the proof is finished in this case.

\medskip

Now, we prove the sharpness of $M_{Q,u}$. For some $r>0$, we defined
$\omega_{k}\left(  \xi\right)  $ by%
\begin{equation}
\omega_{k}\left(  \xi\right)  =\left\{
\begin{array}
[c]{c}%
Q^{\frac{1-Q}{Q}}\left(  c_{Q}\right)  ^{-\frac{1}{Q}}k^{\frac{Q-1}{Q}}\text{
\ \ \ \ \ \ \ \ if }\left\vert \xi\right\vert \in\left[  0,re^{-\frac{k}{Q}%
}\right] \\
Q^{\frac{1}{Q}}\left(  c_{Q}\right)  ^{-\frac{1}{Q}}\log\left(  \frac
{r}{\left\vert \xi\right\vert }\right)  k^{-\frac{1}{Q}}\text{ \ if
}\left\vert \xi\right\vert \in\left[  re^{-\frac{k}{Q}},r\right] \\
0\text{ \ \ \ \ \ \ \ \ \ \ \ \ \ \ \ \ \ \ if }\left\vert \xi\right\vert
\in\left[  r,\infty\right]  ,
\end{array}
\right.  \label{add1}%
\end{equation}
where $c_{Q}=\int_{\Sigma}\left\vert x^{\ast}\right\vert ^{Q}d\xi$, $x^{\ast
}=\frac{x}{\left\vert \xi\right\vert }$ and $\Sigma$ is the unit sphere on
$\mathbb{H}^{n}$.

\medskip

We can verify that $\omega_{k}\left(  \xi\right)  \in HW_{0}^{1,Q}\left(
\Omega\right)  $. Actually, from Lemma \ref{lu} we have%
\begin{align*}
\int_{\Omega}\left\vert \nabla_{\mathbb{H}}\omega_{k}\left(  \xi\right)
\right\vert ^{Q}d\xi &  =\int_{\Sigma}\int_{re^{-\frac{k}{Q}}}^{r}\left\vert
Q^{\frac{1}{Q}}\left(  c_{Q}\right)  ^{-\frac{1}{Q}}k^{-\frac{1}{Q}}%
\frac{\left\vert x^{\ast}\right\vert }{\rho\left(  \xi\right)  }\right\vert
^{Q}\rho\left(  \xi\right)  ^{Q-1}d\rho\left(  \xi\right)  d\mu\left(
x^{\ast}\right) \\
&  =\frac{Q}{c_{Q}}\frac{1}{k}c_{Q}\int_{re^{-\frac{k}{Q}}}^{r}\rho^{-1}%
d\rho=1
\end{align*}
and $\omega_{k}\left(  \xi\right)  \rightharpoonup0$ in $HW_{0}^{1,Q}\left(
\Omega\right)  $.

\medskip

Now, for $R=3r$, we define
\begin{equation}
u=\left\{
\begin{array}
[c]{c}%
A\text{ \ \ \ \ \ \ \ \ \ \ \ \ if }\left\vert \xi\right\vert \in\left[
0,\frac{2R}{3}\right] \\
3A-\frac{3A}{R}\left\vert \xi\right\vert \text{ \ \ \ if }\left\vert
\xi\right\vert \in\left[  \frac{2R}{3},R\right] \\
0\text{ \ \ \ \ \ \ \ \ \ \ \ \ \ if }\left\vert \xi\right\vert \in\left[
R\ ,+\infty\right]  ,
\end{array}
\right.  \label{add2}%
\end{equation}
where $A>0$ is chosen in such a way that\ $\left\Vert \nabla_{\mathbb{H}%
}u\right\Vert _{L^{Q}\left(  \Omega\right)  }=$\ $\delta<1$. Defining
\begin{equation}
u_{k}=u+\left(  1-\delta^{Q}\right)  ^{1/Q}\omega_{k}. \label{add3}%
\end{equation}
Observing that $\nabla_{\mathbb{H}}u$ and $\nabla_{\mathbb{H}}\omega_{k}$ have
disjoints supports, we have
\begin{align*}
\left\Vert \nabla_{\mathbb{H}}u_{k}\right\Vert _{L^{Q}\left(  \Omega\right)
}^{Q}  &  =\left\Vert \nabla_{\mathbb{H}}u\right\Vert _{L^{Q}\left(
\Omega\right)  }^{Q}+\left(  1-\delta^{Q}\right) \\
&  =1,
\end{align*}
moreover, $u_{k}\rightharpoonup u$ in $HW_{0}^{1,Q}\left(  \Omega\right)  $.

Consequently,
\begin{align*}
&  \int_{\Omega}\frac{\exp\left(  \alpha_{Q,\beta}M_{Q,u}u_{k}^{\frac{Q}{Q-1}%
}\right)  }{\rho\left(  \xi\right)  ^{\beta}}d\xi=\int_{\Omega}\frac
{\exp\left(  \frac{\alpha_{Q,\beta}u_{k}^{\frac{Q}{Q-1}}}{\left(  1-\delta
^{Q}\right)  ^{1/\left(  Q-1\right)  }}\right)  }{\rho\left(  \xi\right)
^{\beta}}d\xi\\
&  \geq\int_{B_{r\exp}\left(  -\frac{k}{Q}\right)  }\frac{\exp\left(
\frac{\alpha_{Q,\beta}\left(  A+\left(  1-\delta^{Q}\right)  ^{1/Q}\omega
_{k}\right)  ^{\frac{Q}{Q-1}}}{\left(  1-\delta^{Q}\right)  ^{1/\left(
Q-1\right)  }}\right)  }{\rho\left(  \xi\right)  ^{\beta}}d\xi\\
&  =\int_{B_{r\exp}\left(  -\frac{k}{Q}\right)  }\frac{\exp\left(
\alpha_{Q,\beta}\left(  C+\omega_{k}\right)  ^{\frac{Q}{Q-1}}\right)  }%
{\rho\left(  \xi\right)  ^{\beta}}d\xi\\
&  \geq\exp\left(  \left(  C^{\prime}+\left(  \left(  1-\frac{\beta}%
{Q}\right)  k\right)  ^{\frac{Q-1}{Q}}\right)  ^{\frac{Q}{Q-1}}\right)
\int_{B_{r\exp}\left(  -\frac{k}{Q}\right)  }\frac{1}{\rho\left(  \xi\right)
^{\beta}}d\xi\\
&  \geq C^{\prime\prime}\exp\left(  \left(  C^{\prime}+\left(  \left(
1-\frac{\beta}{Q}\right)  k\right)  ^{\frac{Q-1}{Q}}\right)  ^{\frac{Q}{Q-1}%
}-\left(  1-\frac{\beta}{Q}\right)  k\right)  \rightarrow\infty,
\end{align*}
for some positive constant $C,C^{\prime},C^{\prime\prime}$, and the theorem is finished.
\end{proof}

\subsection{\bigskip Concentration Compactness Principle for the whole space
$\mathbb{H}^{n}$\label{domain infinite}}

In order to prove Theorem \ref{Theorem2}, we need the following

\begin{lemma}
\label{do1}Let $\left\{  u_{k}\right\}  $ be a sequence in $HW^{1,Q}\left(
\mathbb{H}^{n}\right)  $ strongly convergent. Then there exist a subsequence
$\left\{  u_{k_{i}}\right\}  $ of $\left\{  u_{k}\right\}  $ and
$\omega\left(  \xi\right)  \in HW^{1,Q}\left(  \mathbb{H}^{n}\right)  $ such
that $\left\vert u_{k_{i}}\right\vert \leq$ $\omega\left(  \xi\right)  $
almost everywhere on $\mathbb{H}^{n}$.
\end{lemma}

\begin{proof}
The proof is similar as \cite[Proposition 1]{do}, and we omit it.
\end{proof}

Now, we give the

\begin{proof}
[Proof of Theorem \ref{Theorem2}]As in the proof of Theorem \ref{Theorem1}, we
split the proof into two cases.

Case 1: \ $\left\Vert u\right\Vert _{HW^{1,Q}\left(  \mathbb{H}^{n}\right)
}<1.$\ We assume by contradiction for some $p_{1}<\tilde{M}_{Q,u}$, we have%
\[
\underset{k}{\sup}\int_{\mathbb{H}^{n}}\frac{\Phi\left(  \alpha_{Q,\beta}%
p_{1}u_{k}^{\frac{Q}{Q-1}}\right)  }{\rho\left(  \xi\right)  ^{\beta}}%
d\xi=+\infty.
\]
Set
\[
A\left(  u_{k}\right)  =2^{-\frac{1}{Q\left(  Q-1\right)  }}\left\Vert
u_{k}\right\Vert _{L^{Q}\left(  \mathbb{H}^{n}\right)  }\text{, }\Omega\left(
u_{k}\right)  =\left\{  \xi\in\mathbb{H}^{n}:u_{k}\left(  \xi\right)
>A\left(  u_{k}\right)  \right\}
\]
and
\[
\Omega_{L}^{k}=\left\{  \xi\in\mathbb{H}^{n},u_{k}\left(  \xi\right)  \geq
L\right\}  ,
\]
where $L$ is some constant which will be determined later. \ We can
easily verify that%
\[
A\left(  u_{k}\right)  <1\ \text{and }\left\vert \Omega\left(  u_{k}\right)
\right\vert \leq2^{\frac{1}{Q-1}}.
\]
Now, we write
\[
\int_{\mathbb{H}^{n}}\frac{\Phi\left(  \alpha_{Q,\beta}p_{1}u_{k}^{\frac
{Q}{Q-1}}\right)  }{\rho\left(  \xi\right)  ^{\beta}}d\xi=\left(  \int
_{\Omega\left(  u_{k}\right)  }+\int_{\mathbb{H}^{n}\backslash\Omega\left(
u_{k}\right)  }\right)  \frac{\Phi\left(  \alpha_{Q,\beta}p_{1}u_{k}^{\frac
{Q}{Q-1}}\right)  }{\rho\left(  \xi\right)  ^{\beta}}d\xi,
\]
Similar to  the proof of \cite[Theorem 1.6]{Lam}, we can show that
\[
\int_{\mathbb{H}^{n}\backslash\Omega\left(  u_{k}\right)  }\frac{\Phi\left(
\alpha_{Q,\beta}p_{1}u_{k}^{\frac{Q}{Q-1}}\right)  }{\rho\left(  \xi\right)
^{\beta}}d\xi\leq c\left(  p_{1},Q,\beta\right)  .
\]
Therefore, we have
\[
\underset{k}{\sup}\int_{\Omega\left(  u_{k}\right)  }\frac{\Phi\left(
\alpha_{Q,\beta}p_{1}u_{k}^{\frac{Q}{Q-1}}\right)  }{\rho\left(  \xi\right)
^{\beta}}d\xi=\infty
\]

Let $v_{k}=u_{k}-L$. Then for any $\varepsilon>0$, one has
\begin{equation}
u_{k}^{\frac{Q}{Q-1}}\leq\left(  1+\varepsilon\right)  v_{k}^{\frac{Q}{Q-1}%
}+c\left(  \varepsilon,Q\right)  L^{\frac{Q}{Q-1}}. \label{a2}%
\end{equation}
By (\ref{a2}), we have
$$
\int_{\Omega\left(  u_{k}\right)  }\frac{\Phi\left(  \alpha_{Q,\beta}%
p_{1}u_{k}^{\frac{Q}{Q-1}}\right)  }{\rho\left(  \xi\right)  ^{\beta}}d\xi
  \leq  \exp\left(  \alpha_{Q,\beta}p_{1}c\left(  \varepsilon,Q\right)
L^{\frac{Q}{Q-1}}\right)  \int_{\Omega\left(  u_{k}\right)  }\frac{\exp\left(
\left(  1+\varepsilon\right)  \alpha_{Q,\beta}p_{1}v_{k}^{\frac{Q}{Q-1}%
}\right)  }{\rho\left(  \xi\right)  ^{\beta}}d\xi.
$$
Thus
\[
\underset{k}{\sup}\int_{\Omega\left(  u_{k}\right)  }\frac{\exp\left(
\alpha_{Q}\bar{p}_{1}v_{k}^{\frac{Q}{Q-1}}\right)  }{\rho\left(  \xi\right)
^{\beta}}d\xi=\infty,
\]
where $\bar{p}_{1}=\left(  1+\varepsilon\right)  p_{1}<\tilde{M}_{Q,u}$.

Since $\left\vert \Omega\left(  u_{k}\right)  \right\vert \leq2^{\frac{1}%
{Q-1}}$, we have%
\[
\underset{k}{\sup}\int_{\Omega_{L}^{k}}\frac{\exp\left(  \alpha_{Q,\beta}%
\bar{p}_{1}v_{k}^{\frac{Q}{Q-1}}\right)  }{\rho\left(  \xi\right)  ^{\beta}%
}d\xi=\infty.
\]

Now, we define
\[
T^{L}\left(  u\right)  =\min\left\{  L,u\right\}  \text{ and }T_{L}\left(
u\right)  =u-T^{L}\left(  u\right)  \text{.}%
\]
and choose $L$ such that%
\begin{equation}
\frac{1-\left\Vert u\right\Vert _{HW^{1,Q}\left(  \mathbb{H}^{n}\right)  }%
^{Q}}{1-\left\Vert T^{L}u\right\Vert _{HW^{1,Q}\left(  \mathbb{H}^{n}\right)
}^{Q}}>\left(  \frac{\bar{p}_{1}}{\tilde{M}_{Q,u}}\right)  ^{Q-1}. \label{a4}%
\end{equation}
We claim that
\[
\underset{k}{\lim\sup}\int_{\Omega_{L}^{k}}\left\vert \nabla_{\mathbb{H}}%
v_{k}\right\vert ^{Q}d\xi<\left(  \frac{1}{\bar{p}_{1}}\right)  ^{Q-1}.
\]
If not, up to a subsequence, one has
\begin{equation}
\int_{\Omega_{L}^{k}}\left\vert \nabla_{\mathbb{H}}v_{k}\right\vert ^{Q}%
d\xi=\int_{\mathbb{H}^{n}}\left\vert \nabla_{\mathbb{H}}T_{L}u_{k}\right\vert
^{Q}d\xi\geq\left(  \frac{1}{\bar{p}_{1}}\right)  ^{Q-1}+o_{k}\left(
1\right)  . \label{a3}%
\end{equation}
\ Thus,
\begin{align*}
&  \left(  \frac{1}{\bar{p}_{1}}\right)  ^{Q-1}+\int_{\mathbb{H}^{n}%
}\left\vert \nabla_{\mathbb{H}}T^{L}u_{k}\right\vert ^{Q}d\xi+\int
_{\mathbb{H}^{n}}\left\vert T^{L}u_{k}\right\vert ^{Q}d\xi+o_{k}\left(
1\right) \\
&  \leq\left(  \frac{1}{\bar{p}_{1}}\right)  ^{Q-1}+\int_{\mathbb{H}^{n}%
}\left\vert \nabla_{\mathbb{H}}T^{L}u_{k}\right\vert ^{Q}d\xi+\int
_{\mathbb{H}^{n}}\left\vert u_{k}\right\vert ^{Q}d\xi+o_{k}\left(  1\right) \\
&  \leq\int_{\mathbb{H}^{n}}\left\vert \nabla_{\mathbb{H}}T_{L}u_{k}%
\right\vert ^{Q}d\xi+\int_{\mathbb{H}^{n}\backslash\Omega_{L}^{k}}\left\vert
\nabla_{\mathbb{H}}u_{k}\right\vert ^{Q}d\xi+\int_{\mathbb{H}^{n}}\left\vert
u_{k}\right\vert ^{Q}d\xi\\
&  =\int_{\Omega_{L}^{k}}\left\vert \nabla_{\mathbb{H}}u_{k}\right\vert
^{Q}d\xi+\int_{\Omega\backslash\Omega_{L}^{k}}\left\vert \nabla_{\mathbb{H}%
}u_{k}\right\vert ^{Q}d\xi+\int_{\mathbb{H}^{n}}\left\vert u_{k}\right\vert
^{Q}d\xi=1.
\end{align*}
For $L>0$ fixed, $T^{L}u_{k}$ is also bounded in $HW^{1,Q}\left(
\mathbb{H}^{n}\right)  $. Hence, up to a subsequence, $T^{L}u_{k}%
\rightharpoonup T^{L}u$ in $HW^{1,Q}\left(  \mathbb{H}^{n}\right)  $ and
$T^{L}u_{k}\rightarrow T^{L}u$ almost everywhere on $\mathbb{H}^{n}$. By the
lower semicontinuity of the norm in $HW^{1,Q}\left(  \mathbb{H}^{n}\right)  $
and the above inequality, we have%
\[
\bar{p}_{1}\geq\frac{1}{\left(  1-\underset{k\rightarrow\infty}{\lim\inf
}\left\Vert T^{L}u_{k}\right\Vert _{HW^{1,Q}\left(  \mathbb{H}^{n}\right)
}^{Q}\right)  ^{\frac{1}{Q-1}}}\geq\frac{1}{\left(  1-\left\Vert
T^{L}u\right\Vert _{HW^{1,Q}\left(  \mathbb{H}^{n}\right)  }^{Q}\right)
^{\frac{1}{Q-1}}},
\]
combining with (\ref{a4}), we have
\[
\bar{p}_{1}\geq\frac{1}{\left(  1-\left\Vert T^{L}u\right\Vert _{HW^{1,Q}%
\left(  \mathbb{H}^{n}\right)  }^{Q}\right)  ^{\frac{1}{Q-1}}}>\frac{\bar
{p}_{1}}{\tilde{M}_{Q,u}}\frac{1}{\left(  1-\left\Vert T^{L}u\right\Vert
_{HW^{1,Q}\left(  \mathbb{H}^{n}\right)  }^{Q}\right)  ^{\frac{1}{Q-1}}}%
=\bar{p}_{1},
\]
which is a contradiction. \ Therefore
\[
\underset{k}{\lim\sup}\int_{\Omega_{L}^{k}}\left\vert \nabla_{\mathbb{H}}%
v_{k}\right\vert ^{Q}d\xi<\left(  \frac{1}{\bar{p}_{1}}\right)  ^{Q-1}.
\]
By the Trudinger-Moser inequality (\ref{moser-lu}), we have $\ \ $%
\[
\underset{k}{\sup}\int_{\Omega_{L}^{k}}\frac{\exp\left(  \alpha_{Q,\beta}%
\bar{p}_{1}v_{k}^{\frac{Q}{Q-1}}\right)  }{\rho\left(  \xi\right)  ^{\beta}%
}d\xi<\infty,
\]
which is also a contradiction. The proof is finished in this case.

\medskip

Case 2: $\left\Vert u\right\Vert _{HW^{1,Q}\left(  \mathbb{H}^{n}\right)  }%
=1$. Since $HW^{1,Q}\left(  \mathbb{H}^{n}\right)  $ is uniformly convex
Banach space and $u_{k}\rightharpoonup u$ weakly in $HW^{1,Q}\left(
\mathbb{H}^{n}\right)  $, by Radon's Theorem, we have $u_{k}\rightarrow u$
strongly in $HW^{1,Q}\left(  \mathbb{H}^{n}\right)  $. Using Lemma \ref{do1},
there exists some $\omega\left(  \xi\right)  \in HW^{1,Q}\left(
\mathbb{H}^{n}\right)  $, such that up to a subsequence, $\left\vert
u_{k}\right\vert \leq\omega\left(  \xi\right)  \ $a.e. in $\mathbb{H}^{n}$.
Therefore
\[
\int_{\mathbb{H}^{n}}\frac{\Phi\left(  \alpha_{Q,\beta}p_{1}u_{k}^{\frac
{Q}{Q-1}}\right)  }{\rho\left(  \xi\right)  ^{\beta}}d\xi\leq\int
_{\mathbb{H}^{n}}\frac{\Phi\left(  \alpha_{Q,\beta}p_{1}\omega\left(
\xi\right)  ^{\frac{Q}{Q-1}}\right)  }{\rho\left(  \xi\right)  ^{\beta}}d\xi.
\]

Now, we show%
\begin{equation}
\int_{\mathbb{H}^{n}}\frac{\Phi\left(  \alpha_{Q,\beta}p_{1}\omega\left(
\xi\right)  ^{\frac{Q}{Q-1}}\right)  }{\rho\left(  \xi\right)  ^{\beta}}%
d\xi<\infty.\label{add}%
\end{equation}
Set $\Omega\left(  \omega\right)  =\left\{  \xi\in\mathbb{H}^{n}%
:\omega>1\right\}  $, we have
\begin{align*}
\int_{\mathbb{H}^{n}}\left\vert \omega\left(  \xi\right)  \right\vert ^{Q}d\xi
&  \geq\int_{\Omega\left(  \omega\right)  }\left\vert \omega\left(
\xi\right)  \right\vert ^{Q}d\xi\\
&  \geq\left\vert \Omega\left(  \omega\right)  \right\vert .
\end{align*}
Similar as \cite{Lam}, we can derive
\[
\int_{\mathbb{H}^{n}\backslash\Omega\left(  \omega\right)  }\frac{\Phi\left(
\alpha_{Q,\beta}p_{1}\omega\left(  \xi\right)  ^{\frac{Q}{Q-1}}\right)  }%
{\rho\left(  \xi\right)  ^{\beta}}d\xi<C\left(  p_{1},Q,\beta\right)  .
\]
Now, we only need to show
\[
\int_{\Omega\left(  \omega\right)  }\frac{\exp\left(  \alpha_{Q,\beta}%
p_{1}\omega\left(  \xi\right)  ^{\frac{Q}{Q-1}}\right)  }{\rho\left(
\xi\right)  ^{\beta}}d\xi<\infty.
\]
Let $\omega^{\ast}\left(  \xi\right)  $ be the non-increasing rearrangement of
$\omega\left(  \xi\right)  $ in $\Omega\left(  \omega\right)  $. Then \
\[
\int_{\Omega\left(  \omega\right)  }\frac{\exp\left(  \alpha_{Q,\beta}%
p_{1}\omega\left(  \xi\right)  ^{\frac{Q}{Q-1}}\right)  }{\rho\left(
\xi\right)  ^{\beta}}d\xi=\int_{B_{R}}\frac{\exp\left(  \alpha_{Q,\beta}%
p_{1}\omega^{\ast}\left(  \xi\right)  ^{\frac{Q}{Q-1}}\right)  }{\rho\left(
\xi\right)  ^{\beta}}d\xi,
\]
where $\left\vert B_{R}\right\vert =$ $\left\vert \Omega\left(  \omega\right)
\right\vert $. We introduce the variable $t$ by $\ \rho\left(  \xi\right)
^{Q}=R^{Q}e^{-t}$, and set
\[
\varphi\left(  t\right)  =Q^{1-\frac{1}{Q}}c_{Q}^{\frac{1}{Q}}\omega^{\ast
}\left(  \xi\right)  .
\]
Then by Lemma \ref{lu} and the result of Manfredi and Vera De Serio
\cite{Manfredi}\ that there exists a constant $c\geq1$ depending only on $Q$
such that,
\[
\int_{0}^{\infty}\left\vert \varphi^{\prime}\left(  t\right)  \right\vert
^{Q}dt=\int_{B_{R}}\left\vert \nabla_{\mathbb{H}}\omega^{\ast}\left(
\xi\right)  \right\vert ^{Q}d\xi\leq c\int_{\Omega\left(  \omega\right)
}\left\vert \nabla_{\mathbb{H}}\omega\left(  \xi\right)  \right\vert ^{Q}%
d\xi<\infty.
\]
Moreover, we have
\begin{align*}
\int_{\Omega\left(  \omega\right)  }\frac{\exp\left(  \alpha_{Q,\beta}%
p_{1}\omega\left(  \xi\right)  ^{\frac{Q}{Q-1}}\right)  }{\rho\left(
\xi\right)  ^{\beta}}d\xi &  \leq\int_{\Omega\left(  \omega\right)  }%
\frac{\exp\left(  \alpha_{Q,\beta}p_{1}\omega^{\ast}\left(  \xi\right)
^{\frac{Q}{Q-1}}\right)  }{\rho\left(  \xi\right)  ^{\beta}}d\xi\\
&  =\left\vert \Omega\left(  \omega\right)  \right\vert \cdot R^{-\beta}%
\int_{0}^{\infty}\exp\left(  \left(  1-\frac{\beta}{Q}\right)  \left(
p_{1}\varphi\left(  t\right)  ^{\frac{Q}{Q-1}}-t\right)  \right)  dt.
\end{align*}
This follows from the Hardy-Littlewood inequality implies by noticing that the
rearrangement of $\rho\left(  \xi\right)  ^{-\beta}$ is just itself.

Since $\int_{0}^{\infty}\left\vert \varphi^{\prime}\left(  t\right)
\right\vert ^{Q}dt<\infty$, then for all $\varepsilon>0$, there exists
$T=T\left(  \varepsilon\right)  $ such that
\[
\int_{T}^{\infty}\left\vert \varphi^{\prime}\left(  t\right)  \right\vert
^{Q}dt<\varepsilon^{Q}.
\]
Hence, by H\"{o}lder's inequality
\begin{align*}
\varphi\left(  t\right)   &  =\varphi\left(  T\right)  +\int_{T}^{t}%
\varphi^{\prime}\left(  t\right)  dt\\
&  \leq\varphi\left(  T\right)  +\left(  \int_{T}^{t}\left\vert \varphi
^{\prime}\left(  t\right)  \right\vert ^{Q}dt\right)  ^{\frac{1}{Q}}%
\cdot\left\vert t-T\right\vert ^{\frac{Q-1}{Q}}\\
&  \leq\varphi\left(  T\right)  +\varepsilon\left\vert t-T\right\vert
^{\frac{Q-1}{Q}}.
\end{align*}
There exists $\tilde{T}$ such that
\[
p_{1}\varphi\left(  t\right)  ^{\frac{Q}{Q-1}}\leq\frac{t}{2}\text{ for all
}t>\tilde{T}.
\]
Therefore $\int_{\Omega\left(  \omega\right)  }\frac{\exp\left(
\alpha_{Q,\beta}p_{1}\omega\left(  \xi\right)  ^{\frac{Q}{Q-1}}\right)  }%
{\rho\left(  \xi\right)  ^{\beta}}d\xi<\infty$, and the proof is finished in
this case.

Now, we prove the sharpness of $\tilde{M}_{Q,u}$. For some $r>0$ and $R=3r$,
we define $\ \omega_{k}\left(  \xi\right)  ,u\in HW^{1,Q}\left(
\mathbb{H}^{n}\right)  $ as (\ref{add1}),(\ref{add2}), respectively. The
constant $A$ is chosen in such a way that $\left\Vert u\right\Vert
_{HW^{1,Q}\left(  \mathbb{H}^{n}\right)  }=\delta<1$. Defining
\[
u_{k}=u+\left(  1-\delta^{Q}\right)  ^{1/Q}\omega_{k}.
\]
We can easily verify that
\begin{equation}
\left\Vert \omega_{k}\right\Vert _{L^{p}\left(  \mathbb{H}^{n}\right)
}\rightarrow0\text{, for any }p\geq1,\label{add4}%
\end{equation}%
\[
\left\Vert \nabla_{\mathbb{H}}u_{k}\right\Vert _{L^{Q}\left(  \mathbb{H}%
^{n}\right)  }^{Q}=\left\Vert \nabla_{\mathbb{H}}u\right\Vert _{L^{Q}\left(
\mathbb{H}^{n}\right)  }^{Q}+\left(  1-\delta^{Q}\right)  ,
\]
and
\[
u_{k}\rightharpoonup u\text{ weakly in }HW^{1,Q}\left(  \mathbb{H}^{n}\right)
.
\]
Moreover, from (\ref{add4}) we have
\begin{align*}
\int_{\mathbb{H}^{n}}\left\vert u_{k}\right\vert ^{Q}d\xi &  =\int
_{\mathbb{H}^{n}}\left\vert u+\left(  1-\delta^{Q}\right)  ^{1/Q}\omega
_{k}\right\vert ^{Q}d\xi\\
&  =\int_{\mathbb{H}^{n}}\left\vert u\right\vert ^{Q}d\xi+\xi_{k},
\end{align*}
where $\xi_{k}\rightarrow O\left(  \left(  \frac{1}{k}\right)  ^{1/Q}\right)
$, and then we have $\left\Vert u_{k}\right\Vert _{HW^{1,Q}\left(
\mathbb{H}^{n}\right)  }^{Q}=1+\xi_{k}$.  Set $v_{k}=\frac{u_{k}}{\left(
1+\xi_{k}\right)  ^{1/Q}}$, we have
\[
v_{k}\rightharpoonup u\text{ weakly in }HW^{1,Q}\left(  \mathbb{H}^{n}\right)
\text{ with }\left\Vert v_{k}\right\Vert _{HW^{1,Q}\left(  \mathbb{H}%
^{n}\right)  }^{Q}=1\text{. }%
\]
Consequently, for any $\varepsilon_{0}>0$ and $p=\left(  1+\varepsilon
_{0}\right)  \tilde{M}_{Q,u}$, one has
\begin{align*}
&  \int_{\mathbb{H}^{n}}\frac{\Phi\left(  \alpha_{Q,\beta}\tilde{M}_{Q,u}%
v_{k}^{\frac{Q}{Q-1}}\right)  }{\rho\left(  \xi\right)  ^{\beta}}d\xi\\
&  \geq\int_{B_{r\exp\left(  -\frac{k}{Q}\right)  }}\frac{1}{\rho\left(
\xi\right)  ^{\beta}}\exp\left(  \frac{\left(  1+\varepsilon_{0}\right)
\alpha_{Q,\beta}v_{k}^{\frac{Q}{Q-1}}}{\left(  1-\delta^{Q}\right)
^{1/\left(  Q-1\right)  }}\right)  d\xi
\end{align*}
(using the fact that $\xi_{k}\rightarrow0$)%
\begin{align*}
&  \geq\int_{B_{r\exp\left(  -\frac{k}{Q}\right)  }}\frac{1}{\rho\left(
\xi\right)  ^{\beta}}\exp\left(  \frac{\alpha_{Q,\beta}\left(  \left(
1+\varepsilon_{0}^{\prime}\right)  \left(  A+\left(  1-\delta^{Q}\right)
^{1/Q}\omega_{k}\right)  \right)  ^{\frac{Q}{Q-1}}}{\left(  1-\delta
^{Q}\right)  ^{1/\left(  Q-1\right)  }}\right)  d\xi\\
&  =\int_{B_{r\exp\left(  -\frac{k}{Q}\right)  }}\frac{1}{\rho\left(
\xi\right)  ^{\beta}}\exp\left(  \alpha_{Q,\beta}\left[  \left(
1+\varepsilon_{0}^{\prime}\right)  \left(  C+\omega_{k}\right)  \right]
^{\frac{Q}{Q-1}}\right)  d\xi\\
&  \geq\int_{B_{r\exp\left(  -\frac{k}{Q}\right)  }}\frac{1}{\rho\left(
\xi\right)  ^{\beta}}\exp\left(  \left(  1-\frac{\beta}{Q}\right)  \left[
\left(  1+\varepsilon_{0}^{\prime}\right)  \left(  C^{\prime}+k^{\frac{Q-1}%
{Q}}\right)  \right]  ^{\frac{Q}{Q-1}}\right)  d\xi\\
&  \geq C^{\prime\prime}\exp\left(  \left[  \left(  1-\frac{\beta}{Q}\right)
\left(  1+\varepsilon_{0}^{\prime}\right)  \left(  C^{\prime}+k^{\frac{Q-1}%
{Q}}\right)  \right]  ^{\frac{Q}{Q-1}}-\left(  1-\frac{\beta}{Q}\right)
k\right)  \rightarrow\infty
\end{align*}
for some positive constant $\varepsilon_{0}^{\prime},C,C^{\prime}%
,C^{\prime\prime}$, and the theorem is finished.
\end{proof}
\begin{remark}
\label{sharpness} The sequence $\left\{  v_{k}\right\}$ is not enough to show that the supremum (\ref{suprem}) is
infinite when\  $p=\tilde{M}%
_{Q,u}$. Actually, we have
\begin{align*}
& \int_{B_{r\exp\left(  -\frac{k}{Q}\right)  }}\frac{\Phi\left(
\alpha_{Q,\beta}\tilde{M}_{Q,u}v_{k}^{\frac{Q}{Q-1}}\right)  }{\rho\left(
\xi\right)  ^{\beta}}d\xi=\int_{B_{r\exp\left(  -\frac{k}{Q}\right)  }}%
\frac{1}{\rho\left(  \xi\right)  ^{\beta}}\exp\left(  \frac{\alpha_{Q,\beta
}v_{k}^{\frac{Q}{Q-1}}}{\left(  1-\delta^{Q}\right)  ^{1/\left(  Q-1\right)
}}\right)  d\xi\\
& \leq\int_{B_{r\exp\left(  -\frac{k}{Q}\right)  }}\frac{1}{\rho\left(
\xi\right)  ^{\beta}}\exp\left(  \frac{\alpha_{Q,\beta}\left(  \left(  \left(
1+\xi_{k}\right)  ^{-1/Q}\right)  \left(  A+\left(  1-\delta^{Q}\right)
^{1/Q}\omega_{k}\right)  \right)  ^{\frac{Q}{Q-1}}}{\left(  1-\delta
^{Q}\right)  ^{1/\left(  Q-1\right)  }}\right)  d\xi\\
& \leq\int_{B_{r\exp\left(  -\frac{k}{Q}\right)  }}\frac{1}{\rho\left(
\xi\right)  ^{\beta}}\exp\left(  \left(  1-\frac{\beta}{Q}\right)  \left[
\left(  \left(  1+\xi_{k}\right)  ^{-1/Q}\right)  \left(  C^{\prime}%
+k^{\frac{Q-1}{Q}}\right)  \right]  ^{\frac{Q}{Q-1}}\right)  d\xi\\
& \left(  \text{since }1-\left(  1+\xi_{k}\right)  ^{-1/Q}=O\left(  \left(
\frac{1}{k}\right)  ^{1/Q}\right)  \right)  \\
& \leq\int_{B_{r\exp\left(  -\frac{k}{Q}\right)  }}\frac{1}{\rho\left(
\xi\right)  ^{\beta}}\exp\left(  \left(  1-\frac{\beta}{Q}\right)  \left[
\left(  1-C^{\prime\prime}\left(  \frac{1}{k}\right)  ^{1/Q}\right)  \left(
C^{\prime}+k^{\frac{Q-1}{Q}}\right)  \right]  ^{\frac{Q}{Q-1}}\right)  d\xi\\
& \leq\int_{B_{r\exp\left(  -\frac{k}{Q}\right)  }}\frac{1}{\rho\left(
\xi\right)  ^{\beta}}\exp\left(  \left(  1-\frac{\beta}{Q}\right)  \left[
k^{\frac{Q-1}{Q}}-C^{\prime\prime}k^{\frac{Q-2}{Q}}\right]  ^{\frac{Q}{Q-1}%
}\right)  d\xi\\
& \leq C^{\prime\prime\prime\prime}\exp\left(  \left(  1-\frac{\beta}%
{Q}\right)  k\left[  1-C^{\prime\prime\prime}k^{\frac{-1}{Q}}\right]  -\left(
1-\frac{\beta}{Q}\right)  k\right)  \\
& \leq C^{\prime\prime\prime\prime}\exp\left(  -C^{\prime\prime\prime}\left(
1-\frac{\beta}{Q}\right)  k^{^{\frac{Q-1}{Q}}}\right)  <\infty
\end{align*}
for some positive constant $\ C^{\prime}$,$C^{\prime\prime}$,$C^{\prime
\prime\prime}$ and $C^{\prime\prime\prime\prime}$. We remark that this argument is
also suitable for the sequence constructed in \cite{J. M. do} for the sharpness of  $\tilde{M}_{n,u}=\left(  1-\left\Vert u\right\Vert _{W^{1,n}\left(  \mathbb{R}%
^{n}\right)  }\right)  ^{-1/\left(  n-1\right)  }$.
\end{remark}
\medskip

\section{$Q-$sub-Laplacian equations of exponential growth on $\mathbb{H}^{n}$
\label{pde}.}

\bigskip

In this section, let's consider the following nonlinear partial differential
equations on $\mathbb{H}^{n}:$%
\begin{equation}
-\mathrm{div}\left(  \left\vert \nabla_{\mathbb{H}}u\right\vert ^{Q-2}%
\nabla_{\mathbb{H}}u\right)  +V\left(  \xi\right)  \left\vert u\right\vert
^{Q-2}u=\frac{f\left(  u\right)  }{\rho\left(  \xi\right)  ^{\beta}},
\label{eq}%
\end{equation}
where $0\leq\beta<Q$.

The main features of this class of equations (\ref{eq}) are that it is defined
in the whole space $\mathbb{H}^{n}$ and involves critical growth and the
nonlinear operator is $Q$-sub-Laplacian. In spite of a possible failure of the
Palais--Smale (PS) compactness condition, we apply the minimax argument based
on the Concentration-Compactness Principle for $HW^{1,Q}\left(  \mathbb{H}%
^{n}\right)  $ -- Theorem \ref{Theorem2} as in \cite{J. M. do}.

The basic assumptions about $f$ and $V\ $are collected in the following:

\medskip

\textbf{(H1) Assumptions for potential }$V$

\medskip

The potential $V:\mathbb{H}^{n}\rightarrow\mathbb{R}$ is a continuous
potential, and satisfies:

\medskip

(V1) $V$ is a continuous function such that $V(\xi)\geq1\ $ for all $\xi
\in\mathbb{H}^{n}$, and one of the following two conditions:

(V2) $V(\xi)\rightarrow\infty$ as $\rho\left(  \xi\right)  $ $\rightarrow
\infty$; or more generally, for every $M>0$, $\left\vert \left\{  \xi
\in\mathbb{H}^{n}:V(\xi)>M\right\}  \right\vert <\infty$;

(V3) the function $V(\xi)^{-1}$ belongs to $L^{1}(\mathbb{H}^{n})$.

\bigskip

\textbf{(H2) Assumptions for }$f$

\medskip

The function $f\left(  t\right)  :\mathbb{R}\rightarrow\mathbb{R}$ behaves
like $\exp\left(  \alpha t^{\frac{Q}{Q-1}}\right)  $ when $\left\vert
t\right\vert \rightarrow\infty$. Precisely, we assume that $f$ satisfies the
following conditions:

\medskip

(f1) there exist constants $\alpha_{0},b_{0},b_{1}>0$ such that for all
$t\in\mathbb{R}$,%
\[
f\left(  t\right)  \leq b_{0}t^{Q-1}+b_{1}\Phi\left(  \alpha_{0}t^{\frac
{Q}{Q-1}}\right)  ;
\]

\medskip

(f2) there exists $\lambda>Q$ such that for all $\xi\in\mathbb{H}^{n}$ and
$t>0,$%
\[
0<\lambda F\left(  t\right)  :=\lambda\int_{0}^{t}f\left(  s\right)  ds\leq
tf\left(  t\right)  ;
\]

(f3) there exist constant $R_{0},M_{0}>0$ such that for all $\xi\in
\mathbb{H}^{n}$ and $t>R_{0},$%
\[
0\leq F\left(  t\right)  \leq M_{0}f\left(  t\right)  ;
\]

\medskip

(f4) there exist constant $\mu>Q$ and $C_{\mu}$ such that for all $t\geq0,$%
\begin{equation}
f\left(  t\right)  \geq C_{\mu}t^{\mu-1} \label{f4}%
\end{equation}
with $C_{\mu}$ satisfying%
\[
C_{\mu}>\left(  \frac{\alpha_{Q,\beta}}{\alpha_{0}}\right)  ^{\frac{\left(
Q-\mu\right)  \left(  Q-1\right)  }{Q}}\left(  \frac{\mu-Q}{\mu}\right)
^{\frac{\mu-Q}{Q}}\lambda_{\mu}^{\mu/Q},
\]
where%
\[
\lambda_{\mu}:=\underset{u\in\mathcal{S}\backslash0}{\inf}\frac{\left\Vert
u\right\Vert ^{Q}}{\int_{\mathbb{H}^{n}}\frac{\left\vert u\right\vert ^{\mu}%
}{\rho\left(  \xi\right)  ^{\beta}}d\xi}\text{;}%
\]

\medskip

(f5) $\underset{t\rightarrow0^{+}}{\lim\sup}\frac{F\left(  t\right)  }%
{k_{0}\left\vert s\right\vert ^{Q}}<\lambda_{Q}:=\underset{u\in\mathcal{S}%
\backslash0}{\inf}\frac{\left\Vert u\right\Vert ^{Q}}{\int_{\mathbb{H}^{n}%
}\frac{\left\vert u\right\vert ^{Q}}{\rho\left(  \xi\right)  ^{\beta}}d\xi}$.

\medskip

Since we are interested in nonnegative weak solutions, we also suppose the following

(f6) $F\left(  t\right)  =0$ if $t\leq0$.

\medskip

From condition (f1), we conclude that for all $t\in\mathbb{R}\,$,
\[
F\left(  t\right)  \leq b_{2}\Phi\left(  \alpha_{1}t^{\frac{Q}{Q-1}}\right)
\]
for some constant $b_{2},\alpha_{1}>0$. From (\ref{add}), we have
$\frac{F\left(  u\right)  }{\rho\left(  \xi\right)  ^{\beta}}\in L^{1}\left(
\mathbb{H}^{n}\right)  $ for all $u\in\mathcal{S}$. Therefore, the associated
functional to the equation (\ref{eq}) defined by%
\begin{equation}
\newline J\left(  u\right)  =\frac{1}{Q}\int_{\mathbb{H}^{n}}\left(
\left\vert \nabla_{\mathbb{H}}u\right\vert ^{Q}+V\left(  \xi\right)
\left\vert u\right\vert ^{Q}\right)  d\xi-\int_{\mathbb{H}^{n}}\frac{F\left(
u\right)  }{\rho\left(  \xi\right)  ^{\beta}}d\xi\label{energy}%
\end{equation}
is well-defined. Moreover, $J$ is a $C^{1}$ functional on $\mathcal{S}$ with
\[
DJ\left(  u\right)  v=\int_{\mathbb{H}^{n}}\left(  \left\vert \nabla
_{\mathbb{H}}u\right\vert ^{Q-2}\nabla_{\mathbb{H}}u\nabla_{\mathbb{H}%
}v+V\left(  \xi\right)  \left\vert u\right\vert ^{Q-2}uv\right)  d\xi
-\int_{\mathbb{H}^{n}}\frac{f\left(  u\right)  v}{\rho\left(  \xi\right)
^{\beta}}d\xi
\]
for all $v\in\mathcal{S}$. Thus, $DJ\left(  u\right)  =0$ if and only if $u$
$\in\mathcal{S}$ is a weak solution to equation (\ref{eq}).

We define the following space associated with the potential $V$:%
\[
\mathcal{S}=\left\{  u\in HW^{1,Q}\left(  \mathbb{H}^{n}\right)  :\left\Vert
u\right\Vert <\infty\right\}
\]
with the norm $\left\Vert u\right\Vert :=\left(  \int_{\mathbb{H}^{n}}\left(
\left\vert \nabla_{\mathbb{H}}u\right\vert ^{Q}+V\left(  \xi\right)
\left\vert u\right\vert ^{Q}\right)  d\xi\right)  ^{1/Q}$.

From the hypothesis (H1), we have the following compactness result:

\begin{lemma}
If $V\left(  \xi\right)  $ satisfy the hypothesis (H1), then for all $Q\leq
q<\infty$, the embedding
\end{lemma}

\[
\mathcal{S}\hookrightarrow L^{q}\left(  \mathbb{H}^{n}\right)
\]
is compact.

\begin{proof}
The proof is analogous to the proof for the Euclidean case in \cite{Costa,
Rabinowitz}, for the completeness, we give the details here.

 Let $\left\{
u_{k}\right\}  $ be a sequence such that $\left\Vert u_{k}\right\Vert ^{Q}<C$.
 In order to prove this result, we only need to show that
$u_{k}\rightarrow0$ strongly in $L^{q}\left(  \mathbb{H}^{n}\right)  $ for any
$Q\leq q<\infty$, whenever $u_{k}\rightharpoonup0$ weakly in $\mathcal{S}$, as
$k\rightarrow\infty$. \

For any $\varepsilon>0$, from (V2), we can choose some $R>0$ such that
\begin{equation}
V\left(  \xi\right)  \geq\frac{2C}{\varepsilon}\label{addd}%
\end{equation}
for all $\xi$ satisfying $\rho\left(  \xi\right)  \geq R$. Since the embedding $HW^{1,Q}\left(  B_{R}\right)  \hookrightarrow L^{q}\left(
B_{R}\right)  $ is compact, we know $u_{k}\rightarrow0$ strongly in
$L^{q}\left(  B_{R}\right)  $, and then there exists a integer $N>0$ such
that when $k>N$,
\begin{equation}
\int_{B_{R}}\left\vert u_{k}\right\vert ^{Q}d\xi<\frac{\varepsilon}%
{2}.\label{add001}%
\end{equation}
On the other hand, from (\ref{addd}) we have
\[
\frac{2C}{\varepsilon}\int_{\mathbb{H}^{n}\backslash B_{R}}\left\vert
u_{k}\right\vert ^{Q}d\xi\leq\int_{\mathbb{H}^{n}\backslash B_{R}}V\left(
\xi\right)  \left\vert u_{k}\right\vert ^{Q}d\xi<C,
\]
that is
\begin{equation}
\int_{\mathbb{H}^{n}\backslash B_{R}}\left\vert u_{k}\right\vert ^{Q}%
d\xi<\frac{\varepsilon}{2}.\label{add002}%
\end{equation}
Combine (\ref{add001}) and (\ref{add002}) we obtain \
\begin{equation}
\int_{\mathbb{H}^{n}}\left\vert u_{k}\right\vert ^{Q}d\xi\leq\varepsilon
.\label{add000}%
\end{equation}
For any $q<p<\infty$, we define
\[
\lambda=\frac{Q\left(  p-q\right)  }{q\left(  p-Q\right)  }\mathcal{\ }%
\ \text{and }\mu=\frac{p\left(  q-Q\right)  }{q\left(  p-Q\right)  }.
\]
Then $\lambda>0$ and $\mu>0$. $\ $By H\"{o}lder's inequality, the
Trudinger-Moser inequality (\ref{unbounded}) and (\ref{add000}), we have
\begin{align*}
\int_{\mathbb{H}^{n}}\left\vert u_{k}\right\vert ^{q}d\xi &  =\int
_{\mathbb{H}^{n}}\left\vert u_{k}\right\vert ^{q\lambda+\mu q}d\xi\\
&  =\left(  \int_{\mathbb{H}^{n}}\left\vert u_{k}\right\vert ^{q\mu\frac
{p-Q}{q-Q}}d\xi\right)  ^{\frac{q-Q}{p-Q}}\cdot\left(  \int_{\mathbb{H}^{n}%
}\left\vert u_{k}\right\vert ^{q\lambda\frac{p-Q}{p-q}}d\xi\right)
^{\frac{p-q}{p-Q}}\\
&  =\left(  \int_{\mathbb{H}^{n}}\left\vert u_{k}\right\vert ^{p}d\xi\right)
^{\frac{q-Q}{p-Q}}\cdot\left(  \int_{\mathbb{H}^{n}}\left\vert u_{k}%
\right\vert ^{Q}d\xi\right)  ^{\frac{p-q}{p-Q}}\\
&  \leq c\left(  \underset{\left\Vert f\right\Vert _{HW^{1,Q}\left(
\mathbb{H}^{n}\right)  }\leq1}{\underset{f\in HW^{1,Q}\left(  \mathbb{H}%
^{n}\right)  }{\sup}}\int_{\mathbb{H}^{n}}\Phi\left(  \alpha_{Q}f\left(
\xi\right)  ^{\frac{Q}{Q-1}}\right)  d\xi\right)  ^{^{\frac{q-Q}{p-Q}}}%
\cdot\left(  \int_{\mathbb{H}^{n}}\left\vert u_{k}\right\vert ^{Q}d\xi\right)
^{\frac{p-q}{p-Q}}\\
&  \leq c\varepsilon^{\frac{p-q}{p-Q}}.
\end{align*}
The proof is finished.
\end{proof}

\subsection{\textbf{Palais--Smale compactness condition}}

In this subsection, we analyze the compactness of Palais--Smale sequences of
the functional $J$. This is the crucial step in the study of existence results
for equation (\ref{eq}).

First, we recall the definition of Palais-Smale Condition:
\begin{definition}
[Palais--Smale Condition] A sequence$\ \left\{  u_{k}\right\}  \ $ in
$\mathcal{S}$ is called a local Palais--Smale sequence at level $d$ for the
\bigskip functional $J$ ($\left(  PS\right)  _{d}$ sequence), if
\[
J\left(  u_{k}\right)  \rightarrow d\text{ and }\left\Vert DJ\left(
u_{k}\right)  \right\Vert \rightarrow0\text{, as }k\rightarrow\infty,
\]
the functional $J$ is said to satisfy the Palais--Smale condition at level $d$
($\left(  PS\right)  _{d}$ condition), if any $\left(  PS\right)  _{d}$
sequence has a convergent subsequence.
\end{definition}

\begin{lemma}
\label{PS}Under the hypotheses of (H1) and (H2). The functional $J$
satisfies the Palais--Smale condition at level $d$ for any $d<\frac{1}%
{Q}\left(  \frac{\alpha_{Q,\beta}}{\alpha_{0}}\right)  ^{Q-1}$.
\end{lemma}

\begin{proof}
The proof is analogous to the proof of \cite[Proposition 4.1]{J. M. do}. For
the completeness, we give the details here.

Let $\left\{  u_{k}\right\}  $ be a $\left(  PS\right)  _{d}$ sequence for
$J$, that is,
\begin{equation}
J\left(  u_{k}\right)  \rightarrow d\text{ } \label{c}%
\end{equation}
and $\left\vert DJ\left(  u_{k}\right)  v\right\vert \rightarrow0$ for all
$v\in\mathcal{S}$, as$\ $ $k\rightarrow\infty$. Then
\begin{align}
&  \int_{\mathbb{H}^{n}}\frac{f\left(  u_{k}\right)  v}{\rho\left(
\xi\right)  ^{\beta}}d\xi-\int_{\mathbb{H}^{n}}\left(  \left\vert
\nabla_{\mathbb{H}}u_{k}\right\vert ^{Q-2}\nabla_{\mathbb{H}}u_{k}%
\nabla_{\mathbb{H}}v+V\left(  \xi\right)  \left\vert u_{k}\right\vert
^{Q-2}u_{k}v\right)  d\xi\label{tend to 0}\\
&  \leq\varepsilon_{k}\left\Vert v\right\Vert \nonumber
\end{align}
for all $v\in\mathcal{S}$, where $\varepsilon_{k}\rightarrow0$, as
$k\rightarrow\infty$.\ \

Choosing $v=u_{k}$ in (\ref{tend to 0}), by (\ref{c}) we get%
\begin{align*}
&  \int_{\mathbb{H}^{n}}\frac{f\left(  u_{k}\right)  u_{k}}{\rho\left(
\xi\right)  ^{\beta}}-\int_{\mathbb{H}^{n}}\frac{F\left(  u_{k}\right)
\lambda}{\rho\left(  \xi\right)  ^{\beta}}+\frac{\lambda}{Q}\left\Vert
u_{k}\right\Vert ^{Q}-d\lambda-\left\Vert u_{k}\right\Vert ^{Q}\\
&  \leq\varepsilon_{k}\left\Vert u_{k}\right\Vert ,
\end{align*}
From (f2), we have
\[
\frac{\lambda-Q}{Q}\left\Vert u_{k}\right\Vert ^{Q}\leq c+\varepsilon
_{k}\left\Vert u_{k}\right\Vert ,
\]
hence, $u_{k}$ is bounded in $\mathcal{S}$. Since for any $q\geq Q$, the
embedding $\mathcal{S}\hookrightarrow L^{q}\left(  \mathbb{H}^{n}\right)  $ is
compact, we can assume that,%
\begin{align}
u_{k}  &  \rightharpoonup u\text{ weakly in }\mathcal{S},\nonumber\\
u_{k}  &  \rightarrow u\text{ strongly in }L^{q}\left(  \mathbb{H}^{n}\right)
\text{ for any }q\geq Q,\label{limit}\\
u_{k}  &  \rightarrow u\text{ for almost all }\xi\in\mathbb{H}^{n}%
\text{.}\nonumber
\end{align}
From (\ref{limit}), we can verify that
\begin{equation}
\int_{\mathbb{H}^{n}}\frac{\left\vert u_{k}-u\right\vert ^{s}}{\rho\left(
\xi\right)  ^{\beta}}d\xi\rightarrow0\text{ as }k\rightarrow\infty
\label{limite3}%
\end{equation}
for any $s\in\left[  Q,\infty\right)  $ and $\beta\in\left[  0,Q\right)  $.
Actually, for any $R>0$,
\[
\int_{\mathbb{H}^{n}}\frac{\left\vert u_{k}-u\right\vert ^{s}}{\rho\left(
\xi\right)  ^{\beta}}d\xi\leq\left(  \int_{B_{R}}+\int_{\mathbb{H}%
^{n}\backslash B_{R}}\right)  \frac{\left\vert u_{k}-u\right\vert ^{s}}%
{\rho\left(  \xi\right)  ^{\beta}}d\xi=I+II,
\]
by H\"{o}lder's inequality, for any $p,\gamma\in\left(  1,\infty\right)  $, we
have
\begin{align*}
I  &  \leq\left(  \int_{B_{R}}\left\vert u_{k}-u\right\vert ^{sp}d\xi\right)
^{1/p}\cdot\left(  \int_{B_{R}}\frac{1}{\rho\left(  \xi\right)  ^{\beta
p^{\prime}}}d\xi\right)  ^{1/p^{\prime}},\text{ and}\\
II  &  \leq\left(  \int_{\mathbb{H}^{n}\backslash B_{R}}\left\vert
u_{k}-u\right\vert ^{s\gamma}d\xi\right)  ^{1/\gamma}\cdot\left(
\int_{\mathbb{H}^{n}\backslash B_{R}}\frac{1}{\rho\left(  \xi\right)
^{\beta\gamma^{\prime}}}d\xi\right)  ^{1/\gamma^{\prime}},
\end{align*}
where $p^{\prime}=\frac{p}{p-1}$ and $\gamma^{\prime}=\frac{\gamma}{\gamma-1}$.

Choosing $p^{\prime}$ and $\gamma$\ sufficiently closed to $1$ such that
\[
\beta p^{\prime}<1\text{ and }\beta\gamma^{\prime}>Q,
\]
then by (\ref{limit}), we get $\int_{\mathbb{H}^{n}}\frac{\left\vert
u_{k}-u\right\vert ^{s}}{\rho\left(  \xi\right)  ^{\beta}}d\xi\rightarrow0$.

Thanks to \cite[Lemma 5.5 and (6.7)]{Lam}, we have
\begin{equation}
\left\{
\begin{array}
[c]{c}%
\frac{f\left(  u_{k}\right)  }{\rho\left(  \xi\right)  ^{\beta}}%
\rightarrow\frac{f\left(  u\right)  }{\rho\left(  \xi\right)  ^{\beta}}\text{
}\ \ \ \ \ \ \ \ \ \ \ \ \ \ \ \ \ \ \ \ \ \ \ \ \ \ \ \text{in }L_{loc}%
^{1}\left(  \mathbb{H}^{n}\right) \\
\frac{F\left(  u_{k}\right)  }{\rho\left(  \xi\right)  ^{\beta}}%
\rightarrow\frac{F\left(  u\right)  }{\rho\left(  \xi\right)  ^{\beta}}\text{
\ \ \ \ \ \ \ \ \ \ \ \ \ \ \ \ \ \ \ \ \ \ \ \ \ \ \ in }L^{1}\left(
\mathbb{H}^{n}\right) \\
\left\vert \nabla_{\mathbb{H}}u_{k}\right\vert ^{Q-2}\nabla_{\mathbb{H}}%
u_{k}\rightharpoonup\left\vert \nabla_{\mathbb{H}}u\right\vert ^{Q-2}%
\nabla_{\mathbb{H}}u\text{ \ \ \ \ \ weakly in }\left(  L_{loc}^{Q/\left(
Q-1\right)  }\left(  \mathbb{H}^{n}\right)  \right)  ^{2n}.
\end{array}
\right.  \label{add111}%
\end{equation}
From this convergence and passing the limit in (\ref{tend to 0}), we get%
\[
\int_{\mathbb{H}^{n}}\frac{f\left(  u\right)  v}{\rho\left(  \xi\right)
^{\beta}}d\xi-\int_{\mathbb{H}^{n}}\left(  \left\vert \nabla_{\mathbb{H}%
}u\right\vert ^{Q-2}\nabla_{\mathbb{H}}u\nabla_{\mathbb{H}}v+V\left(
\xi\right)  \left\vert u\right\vert ^{Q-2}uv\right)  d\xi=0
\]
for any $v\in C_{0}^{\infty}\left(  \mathbb{H}^{n}\right)  $. By density,
taking $v=u$, we have
\[
\int_{\mathbb{H}^{n}}\frac{f\left(  u\right)  u}{\rho\left(  \xi\right)
^{\beta}}d\xi-\int_{\mathbb{H}^{n}}\left(  \left\vert \nabla_{\mathbb{H}%
}u\right\vert ^{Q}+V\left(  \xi\right)  \left\vert u\right\vert ^{Q}\right)
d\xi=0,
\]
from (f2), we get
\[
\int_{\mathbb{H}^{n}}\left(  \left\vert \nabla_{\mathbb{H}}u\right\vert
^{Q}+V\left(  \xi\right)  \left\vert u\right\vert ^{Q}\right)  d\xi\geq
Q\int_{\mathbb{H}^{n}}\frac{F\left(  u\right)  }{\rho\left(  \xi\right)
^{\beta}}d\xi,
\]
thus, $J\left(  u\right)  \geq0$.

In the following, we prove the strong convergence of $\left\{  u_{k}\right\}
$. For this purpose, we split the proof into two cases:

Case 1: $d=0$. From (\ref{add111}) and (\ref{c}), we have
\[
\left\Vert u\right\Vert ^{Q}\leq\underset{k}{\lim}\left\Vert u_{k}\right\Vert
^{Q}=Q\int_{\mathbb{H}^{n}}\frac{F\left(  u\right)  }{\rho\left(  \xi\right)
^{\beta}}d\xi,\text{ }%
\]
hence $J\left(  u\right)  \leq0$. Therefore $J\left(  u\right)  =0$ and
$\underset{k}{\lim}\left\Vert u_{k}\right\Vert ^{Q}$\ $=\left\Vert
u\right\Vert ^{Q}$. Since $\mathcal{S}$ is a uniformly convex Banach space, by
Radon's Theorem, $u_{k}\rightarrow$ $u$ strongly in $\mathcal{S}$.

Case 2: $d\neq0$. We first claim that $u\neq0$. We assume by contradiction for
$u=0$. By (\ref{add111}), (f2) and (f1), we have $\int_{\mathbb{H}^{n}}%
\frac{F\left(  u_{k}\right)  }{\rho\left(  \xi\right)  ^{\beta}}%
d\xi\rightarrow0$, as $k\rightarrow\infty$. From (\ref{c}) we obtain
\begin{equation}
\int_{\mathbb{H}^{n}}\left(  \left\vert \nabla_{\mathbb{H}}u_{k}\right\vert
^{Q}+V\left(  \xi\right)  \left\vert u_{k}\right\vert ^{Q}\right)
d\xi\rightarrow dQ\text{ as }k\rightarrow\infty.\label{contradicit}%
\end{equation}
Since $d<\frac{1}{Q}\left(  \frac{\alpha_{Q,\beta}}{\alpha_{0}}\right)
^{Q-1}$, we can choose some $q>1$ close to $1$ sufficiently such that
\begin{align*}
&  q\alpha_{0}\left(  \int_{\mathbb{H}^{n}}\left(  \left\vert \nabla
_{\mathbb{H}}u_{k}\right\vert ^{Q}+\left\vert u_{k}\right\vert ^{Q}\right)
d\xi\right)  ^{1/\left(  Q-1\right)  }\\
&  \leq q\alpha_{0}\left(  \int_{\mathbb{H}^{n}}\left(  \left\vert
\nabla_{\mathbb{H}}u_{k}\right\vert ^{Q}+V\left(  \xi\right)  \left\vert
u_{k}\right\vert ^{Q}\right)  d\xi\right)  ^{1/\left(  Q-1\right)  }%
<\alpha_{Q,q\beta}.
\end{align*}
Then, by Lemma \ref{Lemma2}, we have
\begin{equation}
\int_{\mathbb{H}^{n}}\frac{\Phi\left(  q\alpha_{0}u_{k}^{\frac{Q}{Q-1}%
}\right)  }{\rho\left(  \xi\right)  ^{q\beta}}d\xi\leq c.\label{bound}%
\end{equation}

By H\"{o}lder's inequality, combining (f1), (\ref{bound}), (\ref{limit}) and
(\ref{limite3}) we get
\[
\int_{\mathbb{H}^{n}}\frac{f\left(  u_{k}\right)  u_{k}}{\rho\left(
\xi\right)  ^{\beta}}d\xi\leq c\int_{\mathbb{H}^{n}}\frac{u_{k}^{Q}}%
{\rho\left(  \xi\right)  ^{\beta}}d\xi+c\left(  \int_{\mathbb{H}^{n}}%
\frac{\Phi\left(  q\alpha_{0}u_{k}^{\frac{Q}{Q-1}}\right)  }{\rho\left(
\xi\right)  ^{q\beta}}d\xi\right)  ^{1/q}\left(  \int_{\mathbb{H}^{n}}%
u_{k}^{q^{\prime}}d\xi\right)  ^{1/q^{\prime}}\rightarrow0,
\]
where $q^{\prime}=\frac{q}{q-1}$.

On the other hand, since $u_{k}$ is a $\left(  PS\right)  _{d}$ sequence,
$DJ\left(  u_{k}\right)  u_{k}\rightarrow0$, i.e.,%
\[
\int_{\mathbb{H}^{n}}\frac{f\left(  u_{k}\right)  u_{k}}{\rho\left(
\xi\right)  ^{\beta}}d\xi-\int_{\mathbb{H}^{n}}\left(  \left\vert
\nabla_{\mathbb{H}}u_{k}\right\vert ^{Q}+V\left(  \xi\right)  \left\vert
u_{k}\right\vert ^{Q}\right)  d\xi\rightarrow0,
\]
thus, $\int_{\mathbb{H}^{n}}\left(  \left\vert \nabla_{\mathbb{H}}%
u_{k}\right\vert ^{Q}+V\left(  \xi\right)  \left\vert u_{k}\right\vert
^{Q}\right)  d\xi\rightarrow0$, which is a contradiction, and the claim is proved.

Set $\tilde{u}_{k}=\frac{u_{k}}{\left\Vert u_{k}\right\Vert }$ and $\tilde
{u}=\frac{u}{\underset{k}{\lim}\left\Vert u_{k}\right\Vert }$. Then
$\left\Vert u_{k}\right\Vert =1$ and $\tilde{u}_{k}\rightharpoonup\tilde{u}$
weakly in $\mathcal{S}$. If $\left\Vert \tilde{u}\right\Vert =1$, we have
$\underset{k}{\lim}\left\Vert u_{k}\right\Vert =\left\Vert u\right\Vert $, and
then $u_{k}\rightarrow$ $u$ strongly in $\mathcal{S}$.

If $\left\Vert u\right\Vert <1$, \ by (\ref{c}) and (\ref{add111}) and the
fact that $J\left(  u\right)  \geq0$, one has
\[
d+o_{k}\left(  1\right)  >J\left(  u_{k}\right)  -J\left(  u\right)
\rightarrow\frac{1}{Q}\left(  \left\Vert u_{k}\right\Vert ^{Q}-\left\Vert
u\right\Vert ^{Q}\right)  ,
\]
thus,%
\[
\left\Vert u_{k}\right\Vert ^{Q}\left(  1-\left\Vert \frac{u}{\left\Vert
u_{k}\right\Vert}\right\Vert ^{Q}\right)  <Q\left(  d+o_{k}\left(
1\right)  \right)  ,
\]
that is
\[
\left\Vert u_{k}\right\Vert ^{Q/\left(  Q-1\right)  }<\frac{\frac
{\alpha_{Q,\beta}}{\alpha_{0}}+o_{k}\left(  1\right)  }{\left(  1-\left\Vert
\frac{u}{\left\Vert u_{k}\right\Vert}\right\Vert ^{Q}\right)  ^{1/\left(
Q-1\right)  }},
\]
therefore when $k$ large, we can choose some $q>1$ close to $1$ sufficiently
such that%
\begin{equation}
q\alpha_{0}\left\Vert u_{k}\right\Vert ^{Q/\left(  Q-1\right)  }<\frac
{\alpha_{Q,q\beta}}{\left(  1-\left\Vert \tilde{u}\right\Vert ^{Q}\right)
^{1/\left(  Q-1\right)  }}. \label{bound2}%
\end{equation}
By Theorem \ref{Theorem2}, $\frac{\Phi\left(  q\alpha_{0}u_{k}^{\frac{Q}{Q-1}%
}\right)  }{\rho\left(  \xi\right)  ^{q\beta}}$ is bounded in $L^{1}\left(
\mathbb{H}^{n}\right)  $.

By H\"{o}lder's inequality, combining (f1), (\ref{bound2}), (\ref{limit}),
(\ref{limite3}), we get%
\begin{equation}%
\begin{array}
[c]{c}%
\left\vert \int_{\mathbb{H}^{n}}\frac{f\left(  u_{k}\right)  \left(
u_{k}-u\right)  }{\rho\left(  \xi\right)  ^{\beta}}d\xi\right\vert \leq
c\left(  \int_{\mathbb{H}^{n}}\frac{u_{k}^{Q}}{\rho\left(  \xi\right)
^{\beta}}d\xi\right)  ^{\frac{Q-1}{Q}}\left(  \int_{\mathbb{H}^{n}}%
\frac{\left\vert u_{k}-u\right\vert ^{Q}}{\rho\left(  \xi\right)  ^{\beta}%
}d\xi\right)  ^{\frac{1}{Q}}\\
\text{ \ \ \ \ \ \ \ \ \ \ \ \ \ \ \ \ \ \ \ \ \ \ \ \ \ \ \ \ \ \ \ }%
+c\left(  \int_{\mathbb{H}^{n}}\frac{\Phi\left(  q\alpha_{0}u_{k}^{\frac
{Q}{Q-1}}\right)  }{\rho\left(  \xi\right)  ^{q\beta}}d\xi\right)
^{1/q}\left(  \int_{\mathbb{H}^{n}}\left\vert u_{k}-u\right\vert ^{q^{\prime}%
}d\xi\right)  ^{1/q^{\prime}}\rightarrow0.
\end{array}
\label{convergence}%
\end{equation}
Since $DJ\left(  u_{k}\right)  \left(  u_{k}-u\right)  \rightarrow0$, \ from
(\ref{convergence}) we derive%
\begin{equation}
\int_{\mathbb{H}^{n}}\left(  \left\vert \nabla_{\mathbb{H}}u_{k}\right\vert
^{Q-2}\nabla_{\mathbb{H}}u_{k}\nabla_{\mathbb{H}}\left(  u_{k}-u\right)
+V\left(  \xi\right)  \left\vert u_{k}\right\vert ^{Q-2}u_{k}\left(
u_{k}-u\right)  \right)  d\xi\rightarrow0. \label{11}%
\end{equation}
On the other hand, since $u_{k}\rightharpoonup u$ in $\mathcal{S}$, we have
\begin{equation}
\int_{\mathbb{H}^{n}}\left(  \left\vert \nabla_{\mathbb{H}}u\right\vert
^{Q-2}\nabla_{\mathbb{H}}u\nabla_{\mathbb{H}}\left(  u_{k}-u\right)  +V\left(
\xi\right)  \left\vert u\right\vert ^{Q-2}u\left(  u_{k}-u\right)  \right)
d\xi\rightarrow0.\ \label{12}%
\end{equation}
Combining (\ref{11}) and (\ref{12}), we obtain there is a constant $c>0$ so
that%
\[%
\begin{array}
[c]{c}%
\left\Vert u_{k}-u\right\Vert \leq c\int_{\mathbb{H}^{n}}\left(  \left\vert
\nabla_{\mathbb{H}}u_{k}\right\vert ^{Q-2}\nabla_{\mathbb{H}}u_{k}-\left\vert
\nabla_{\mathbb{H}}u\right\vert ^{Q-2}\nabla_{\mathbb{H}}u\right)
\nabla_{\mathbb{H}}\left(  u_{k}-u\right) \\
\text{ \ \ \ \ \ \ \ \ \ \ \ \ \ \ \ \ \ \ \ \ \ \ \ \ \ \ \ }+c\int
_{\mathbb{H}^{n}}V\left(  \xi\right)  \left(  \left\vert u_{k}\right\vert
^{Q-2}u_{k}-\left\vert u\right\vert ^{Q-2}u\right)  \left(  u_{k}-u\right)
d\xi\rightarrow0,
\end{array}
\]
where we have used the inequality $\left(  \left\vert a\right\vert
^{Q-2}a-\left\vert b\right\vert ^{Q-2}b\right)  \left(  a-b\right)
\geq2^{2-Q}\left\vert a-b\right\vert ^{Q}$, for all $a,b\in\mathbb{R}^{2n}$.
The proof is finished.
\end{proof}

From the proof of \cite[Lemma 5.1 and Lemma 5.2.]{Lam}, we have the following
geometric conditions of the mountain-pass theorem:

\begin{lemma}
\label{geometric}Suppose that the hypotheses of (H1) and (H2) hold. Then

(i)\bigskip\ there exists $r,\delta>0$ such that $J\left(  u\right)
\geq\delta$ if $\left\Vert u\right\Vert =r$;

(ii) there exists $e\in\mathcal{S}$ with $\left\Vert e\right\Vert >r$, such
that $J\left(  e\right)  <0$.
\end{lemma}

Now, we define the minimax level by
\[
d_{\infty}=\underset{g\in\Gamma}{\inf}\underset{t\in\left[  0,1\right]  }%
{\max}J\left(  g\left(  t\right)  \right)  ,
\]
where $\Gamma=\left\{  g\in C\left(  \left[  0,1\right]  ,\mathcal{S}\right)
:g\left(  0\right)  =0\text{ and }g\left(  1\right)  <0\right\}  $. From Lemma
\ref{geometric}, we have $d_{\infty}>0$.

\begin{lemma}
\label{c infinit}\bigskip Under the hypotheses of (H1) and (H2). We have
$d_{\infty}<\frac{1}{Q}\left(  \frac{\alpha_{Q,\beta}}{\alpha_{0}}\right)
^{Q-1}$.
\end{lemma}

\begin{proof}
Let $\left\{  v_{k}\right\}  $ be in $\mathcal{S}$ with $\int_{\mathbb{H}^{n}%
}\frac{\left\vert v_{k}\right\vert ^{\mu}}{\rho\left(  \xi\right)  ^{\beta}%
}d\xi=1$ and $\left\Vert v_{k}\right\Vert ^{Q}\rightarrow\lambda_{\mu}$. Then
$\left\{  v_{k}\right\}  $ is bounded in $\mathcal{S}$. Using the compactness
of embedding \ $\mathcal{S}\hookrightarrow L^{q}\left(  \mathbb{H}^{n}\right)
$ for all $q\geq Q$, up to a subsequence, we have%
\begin{align}
v_{k} &  \rightharpoonup v_{0}\text{ weakly in }\mathcal{S},\nonumber\\
v_{k} &  \rightarrow v_{0}\text{ strongly in }L^{q}\left(  \mathbb{H}%
^{n}\right)  \text{ for all }q\in\left[  Q,\infty\right)  \label{limt2}\\
v_{k} &  \rightarrow v_{0}\text{ for almost all }\xi\in\mathbb{H}^{n}%
\text{.}\nonumber
\end{align}
(\ref{limite3}) implies that $\int_{\mathbb{H}^{n}}\frac{\left\vert
v_{0}\right\vert ^{\mu}}{\rho\left(  \xi\right)  ^{\beta}}d\xi=\underset
{k}{\lim}\int_{\mathbb{H}^{n}}\frac{\left\vert v_{k}\right\vert ^{\mu}}%
{\rho\left(  \xi\right)  ^{\beta}}d\xi=1$. By the semicontinuity of the norm
$\left\Vert \cdot\right\Vert $, we infer that
\[
\left\Vert v_{0}\right\Vert ^{Q}\leq\underset{k}{\lim\inf}\left\Vert
v_{k}\right\Vert ^{Q}=\lambda_{\mu}\text{,}%
\]
thus $\lambda_{\mu}$ is attained by $v_{0}$, we may assume that $v_{0}\geq0$.

From (\ref{f4}), we know that $F\left(  t\right)  \geq\frac{C_{\mu}}{\mu
}t^{\mu}$ for some $\mu>Q$. Hence,
\begin{align*}
J\left(  tv_{0}\right)   &  \leq\frac{t^{Q}}{Q}\int_{\mathbb{H}^{n}}\left(
\left\vert \nabla_{\mathbb{H}}v_{0}\right\vert ^{Q}+V\left(  \xi\right)
\left\vert v_{0}\right\vert ^{Q}\right)  d\xi-\frac{C_{\mu}t^{\mu}}{\mu}%
\int_{\mathbb{H}^{n}}\frac{v_{0}^{\mu}}{\rho\left(  \xi\right)  ^{\beta}}%
d\xi\\
&  \rightarrow-\infty,\text{ as }t\rightarrow\infty\text{.}%
\end{align*}
Setting $\tilde{v}_{0}\left(  t\right)  =tt_{0}v_{0}$ with $t_{0}$
sufficiently large, then $\tilde{v}_{0}\left(  t\right)  \in\Gamma$. By (f4),
we have
\begin{align*}
d_{\infty} &  \leq\underset{t\in\left[  0,1\right]  }{\max}J\left(  \tilde
{v}_{0}\left(  t\right)  \right)  \leq\underset{t\in\left[  0,1\right]  }%
{\max}\left(  \frac{\left(  t_{0}t\right)  ^{Q}}{Q}\left\Vert v_{0}\right\Vert
^{Q}-\frac{C_{\mu}\left(  t_{0}t\right)  ^{\mu}}{\mu}\int_{\mathbb{H}^{n}%
}\frac{v_{0}^{\mu}}{\rho\left(  \xi\right)  ^{\beta}}d\xi\right)  \\
&  \leq\underset{t>0}{\max}\left(  \frac{t^{Q}}{Q}\left\Vert v_{0}\right\Vert
^{Q}-\frac{C_{\mu}}{\mu}t^{\mu}\int_{\mathbb{H}^{n}}\frac{v_{0}^{\mu}}%
{\rho\left(  \xi\right)  ^{\beta}}d\xi\right)  \\
&  \leq\underset{t>0}{\max}\left(  \frac{t^{Q}}{Q}\lambda_{\mu}-\frac{C_{\mu}%
}{\mu}t^{\mu}\right)  =\frac{\mu-Q}{Q\mu}\frac{\lambda_{\mu}^{\mu/\left(
\mu-Q\right)  }}{C_{\mu}^{Q/\left(  \mu-Q\right)  }}\\
&  <\frac{1}{Q}\left(  \frac{\alpha_{Q,\beta}}{\alpha_{0}}\right)  ^{Q-1},
\end{align*}
and this completes the proof.
\end{proof}

Finally, we come to the

\begin{proof}
[Proof of Theorem \ref{ground state}]Let $\left\{  u_{k}\right\}  $ be a
sequence in $\mathcal{S}$ such that
\[
J\left(  u_{k}\right)  \rightarrow d_{\infty}\text{ }%
\]
and $DJ\left(  u_{k}\right)  \rightarrow0$. By Lemmas \ref{PS} and
\ref{c infinit}, the sequence $\left\{  u_{k}\right\}  $ converges weakly to a
weak solution $u_{0}$ of (\ref{eq}). \ \ Now, we show $u_{0}>0$ in
$\mathbb{H}^{n}$.

Set $u_{0+}:=\max\left\{  u_{0},0\right\}  $ and $u_{0-}:=\max\left\{
-u_{0},0\right\}  $. \ Since $u_{0}$ satisfies $DJ\left(  u_{0}\right)  =0$,
we have $DJ\left(  u_{0}\right)  u_{0-}=0$, that is,
\[
\left\Vert \ u_{0-}\right\Vert ^{Q}-\int_{\mathbb{H}^{n}}\frac{f\left(
u_{0}\right)  u_{0-}}{\rho\left(  \xi\right)  ^{\beta}}d\xi=0.
\]
On the other hand, from (f6) we have $\int_{\mathbb{H}^{n}}\frac{f\left(
u_{0}\right)  u_{0-}}{\rho\left(  \xi\right)  ^{\beta}}d\xi=0$, and then
$\left\Vert \ u_{0-}\right\Vert ^{Q}=0$. Therefore, $u_{0}\geq0$ on
$\mathbb{H}^{n}$. From $J\left(  u_{0}\right)  =d_{\infty}>0$, we know
$u_{0}\ $is positive on $\mathbb{H}^{n}$.

Now, let%
\[
M_{\infty}:=\underset{u\in\mathcal{P}\backslash0}{\inf}J\left(  u\right)  ,
\]
where $\mathcal{P}:=\left\{  u\in\mathcal{S}:DJ\left(  u\right)  =0\right\}  $.

In order to show that $u_{0}$ is a ground state solution of (\ref{eq}), we
only need to prove $d_{\infty}$ $\leq M_{\infty}$. For any $u\in
$\ $\mathcal{P}\backslash0$, we define $m\left(  t\right)  $ by $m\left(
t\right)  =J\left(  tu\right)  $. Since $J\in C^{1}\left(  \mathcal{S}%
,\mathbb{R}\right)  $, we have $m\left(  t\right)  $ is differentiable and
\[
m^{\prime}\left(  t\right)  =DJ\left(  tu\right)  u=t^{Q-1}\left\Vert
\ u\right\Vert ^{Q}-\int_{\mathbb{H}^{n}}\frac{f\left(  tu\right)  u}%
{\rho\left(  \xi\right)  ^{\beta}}d\xi,
\]
for any $t>0$.

From $DJ\left(  u\right)  u=0$, we derive%
\[
m^{\prime}\left(  t\right)  =t^{Q-1}\int_{\mathbb{H}^{n}}\left(
\frac{f\left(  u\right)  }{u^{Q-1}}-\frac{f\left(  tu\right)  }{\left(
tu\right)  ^{Q-1}}\right)  \frac{u^{Q}}{\rho\left(  \xi\right)  ^{\beta}}%
d\xi.
\]
By (f2), we know $\frac{f\left(  t\right)  }{t^{Q-1}}$ is increasing for all $s>0$.  From this and the fact $m^{\prime}\left(  1\right)  =0$, we know
$m^{\prime}\left(  t\right)  >0$ if $t\in\left(  0,1\right)  $, and
$m^{\prime}\left(  t\right)  <0$ if $t\in\left(  1,\infty\right)  $. Thus,
$J\left(  u\right)  =\underset{t\geq0}{\max}J\left(  tu\right)  $.

Setting $\tilde{u}\left(  t\right)  =tt_{0}u$ with $t_{0}$ sufficiently large,
we get $\tilde{u}\left(  t\right)  \in\Gamma$, and then%
\[
d_{\infty}\leq\underset{t\in\left[  0,1\right]  }{\max}J\left(  \tilde
{u}\left(  t\right)  \right)  \leq\underset{t\geq0}{\max}J\left(  tu\right)
=J\left(  u\right)  .
\]
Therefore, $d_{\infty}$ $\leq M_{\infty}$. The proof is completed.
\end{proof}

{\bf Acknowledgement} The results of this paper were presented 
by the third author at the Workshop in Fourier Analysis in Sanya Mathematical Forum in August, 2016 and by 
 the second author at the AMS special session on  Geometric Aspects of Harmonic Analysis in Maine in September, 2016.

\end{document}